\documentclass[12pt, reqno]{amsart}
\usepackage{mathrsfs}
\usepackage{amsfonts}
\usepackage{xcolor}
\usepackage{enumerate}
\usepackage{amsmath}
\usepackage{amsthm}
\usepackage{amssymb}
\usepackage{cases}
\usepackage{bm}
\usepackage{cite}
\usepackage[pdfstartview=FitH, colorlinks=true,  linkcolor=red, citecolor=blue]{hyperref}
\numberwithin{equation}{section}
\newtheorem{thm}{Theorem}[section]

\newtheorem{lem}{Lemma}[section]
\newtheorem{cor}{Corollary}[section]
\newtheorem{remark}{Remark}[section]

\newtheorem{theoremalph}{Theorem}

\begin{document}
\title[The $L_p$ dual Minkowski problem]
{uniqueness and continuity of the solution to $L_p$ dual Minkowski problem* }

\author{Hejun Wang}

\address{\parbox[l]{1\textwidth}{School of Mathematics and Statistics, Shandong Normal University, Ji'nan, Shandong 250014, China\\
School of Mathematics and
Statistics, Southwest University, Chongqing 400715, China
}}
\email{wanghjmath@sdnu.edu.cn}

\author{Jiazu Zhou**}
\address{\parbox[l]{1\textwidth}{School of Mathematics and
Statistics, Southwest University, Chongqing 400715, China.\\
 College of Science, Wuhan University of Science and Technology, Wuhan, Hubei 430081, China}} \email{zhoujz@swu.edu.cn }


\thanks{*Supported in part by NSFC (No.12071378), China Postdoctoral Science Foundation (No.2020M682222)
and Natural Science Foundation of Shandong (No.ZR2020QA003) }
\thanks{**The corresponding author}

\maketitle

\begin{abstract}
Lutwak, Yang and Zhang \cite{LYZ2018} introduced the $L_p$ dual curvature measure that
unifies several other geometric measures in dual Brunn-Minkowski theory and Brunn-
Minkowski theory. Motivated by works in \cite{LYZ2018}, we consider the uniqueness and continuity
of the solution to the $L_p$ dual Minkowski problem. To extend the important work (Theorem
\ref{uniquepolytope}) of LYZ to the case for general convex bodies, we establish some new
Minkowski-type inequalities which are closely related to the optimization problem associated with
the $L_p$ dual Minkowski problem. When $q< p$, the uniqueness of the solution to the $L_p$ dual
Minkowski problem for general convex bodies is obtained. Moreover, we obtain the
continuity of the solution to the $L_p$ dual Minkowski problem for convex bodies.
\end{abstract}

%%%%%%%%%%%%%%%%%%%%%%%%%%%%%%%%%%%%%%%%%%%%%%%%%%%%%%%%%%%%%%%%%%%%%%%%%%
\vskip 0.5cm
\section{Introduction}
\vskip 0.3cm
%%%%%%%%%%%%%%%%%%%%%%%%%%%%%%%%%%%%%%%%%%%%%%%%%%%%%%%%%%%%%%%%%%%%%%%%%%

A compact convex subset in the Euclidean space $\mathbb{R}^n$ with non-empty interior is called a convex body.
Let $\mathcal{K}_{o}^n$ denote
the set of convex bodies in $\mathbb{R}^n$ containing the origin in their interiors
and $\mathcal{K}_{e}^n$ denote the set of convex bodies in $\mathbb{R}^n$ that are symmetric about the origin.
Let $S^{n-1}$ denote the unit sphere in $\mathbb{R}^n$.

As one of most fundamental problems
in the Brunn-Minkowski theory of convex bodies,
the classical Minkowski problem is concerned with the characterization of the
so-called surface area measure:

For each Borel $\eta \subset S^{n-1}$, the surface area measure $S(K,\cdot)$ of convex body $K$
in $\mathbb{R}^n$ is defined by
\begin{align*}
S(K,\eta)=\mathcal{H}^{n-1}(\nu^{-1}(\eta)),
\end{align*}
where $\nu_K: \partial' K \rightarrow S^{n-1}$ is the Gauss map of $K$, defined on $\partial' K$ , the subset of the boundary $\partial K$ with unique outer unit normal,
and $\mathcal{H}^{n-1}$ is an ($n-1$)-dimensional Hausdorff measure.

It is clear to see that the surface area measure $S(K,\cdot)$ is a Borel measure on the unit sphere $S^{n-1}$.
The classical Minkowski problem characterizing the surface area measure is as follows:
\vskip 0.2cm
\noindent
{\bf The classical Minkowski problem:}
\emph{What are necessary and sufficient conditions
for a finite Borel measure $\mu$ on $S^{n-1}$ so that $\mu$ is the surface area measure
of a convex body in $\mathbb{R}^n$?}
\vskip 0.2cm

The classical Minkowski problem for discrete measure is a particular case, which is called discrete Minkowski problem.
A polytope $P$ in $\mathbb{R}^n$ is the convex hull of a finite set of points in $\mathbb{R}^n$ provided its interiors
$\text{int}~P\neq\emptyset$. The facets of $P$ are the faces whose dimension is $n-1$.
Let $a_1,\cdots,a_k$ be the $(n-1)$-dimensional volumes of all facets of $P$ whose exterior unit normal vectors are
$u_1,\cdots,u_k$. Then $S(P,\cdot)$ is a discrete measure on $S^{n-1}$ which is concentrated on the set $\{u_1,\cdots,u_k\}$
and $S(P,u_i)=a_i$ for $i=1,\cdots,k$.

The discrete Minkowski problem asks for what conditions
for a given discrete finite Borel measure $\mu$ on $S^{n-1}$ to be the surface are measure of a polytope.

Furthermore, if such a polytope exists, is it unique?
Minkowski \cite{Minkowski1897,Minkowski1903} solved the existence and uniqueness of the solution to the discrete Minkowski problem. The classical Minkowski problem was solved by Aleksandrov \cite{Aleksandrov1938} and independently by
Fenchel and Jessen \cite{FenchelJ1938}.

The $L_p$ surface area measure, introduced by Lutwak \cite{Lutwak3},
is an important extension of the classical surface area measure $S(K,\cdot)$ .
The Minkowski problem for the $L_p$ surface area measure is called the $L_p$ Minkowski problem
of which the classical Minkowski problem,
the centro-affine Minkowski problem \cite{ChouWang} and the logarithmic
Minkowski problem \cite{BoroczkyLYZ2}
are special cases.
In recent decades, the $L_p$ surface area measure and its related problems are studied in
\cite{BartheMN,BoroczkyHenk,BoroczkyLYZ1,
HaberlSchuster1,LudwigXiaoZhang,
Lutwak3,LutwakOliker}.
The existence and uniqueness of the solution to
the $L_p$ Minkowski problem were studied in  \cite{Chen,Lutwak3,LYZ3,ChouWang,ChenLiZhu,HaberlLYZ,HugLYZ,BoroczkyLYZ1,BoroczkyLYZ2,
Stancu1,Stancu2,HuangLiuXu,LuWang,BoroczkyHZ,JianLuZhu,Zhu1,Zhu2,Zhu3,Zhu5}.
Zhu \cite{Zhu4} obtained the continuity of the solution to the $L_p$ Minkowski problem for $1<p\neq n$.
When $p=0$ and $0<p<1$, the part results of continuity were obtained in \cite{WangLv,WFZ2019-2}.
The solutions of the $L_p$ Minkowski problem have important
applications to affine isoperimetric inequalities  \cite{Zhang1,LYZ2,CianchiLYZ,HaberlSchuster1,HaberlSchuster2,HaberlSchusterXiao,Wang1}.

As a ``dual" counterpart of the classical Brunn-Minkowski theory for convex bodies, the dual Brunn-Minkowski theory for star bodies also gets rapid development.
A star body $Q\subseteq\mathbb{R}^n$ is a compact star shaped set with respect to the origin whose radial function
$\rho_Q: S^{n-1}\rightarrow [0,+\infty)$, given by $\rho_Q(u)=\max\{\lambda\geq0: \lambda u\in Q\}$ for $u\in S^{n-1}$,
is continuous.
Let $\mathcal{S}_{o}^n$ denote
the set of star bodies (with respect to the origin) in $\mathbb{R}^n$ containing the origin in their interiors
and $\mathcal{S}_{e}^n$ denote the set of those elements of $\mathcal{S}_{o}^n$ that are symmetric about the origin.

As one of the core problems of the dual Brunn-Minkowski theory, the dual Minkowski problem
for dual curvature measure has been the focus of attention.
In \cite{HuangLYZ},
Huang, Lutwak, Yang and Zhang introduced the $q$-th dual curvature measure of a convex body in $\mathcal{K}_{o}^n$.
Recently, the more general version of the dual curvature measure was posed by Lutwak, Yang and Zhang \cite{LYZ2018}:

For any Borel set $\eta\subseteq S^{n-1}$, $q\in \mathbb{R}$ , the $q$-th dual curvature measure $\widetilde{C}_{q}(K,Q,\cdot)$ of $K\in \mathcal{K}_{o}^n$ with respect to $Q\in \mathcal{S}_{o}^n$
is defined  by
\begin{align}
\widetilde{C}_{q}(K,Q,\eta)=\frac{1}{n}\int_{\bm{\alpha}^*_K(\eta)}\rho^q_{K}(u)\rho^{n-q}_{Q}(u)du,
\end{align}
where $\bm{\alpha}^*_K(\eta)$ denotes the set of unit vector $u$
such that an outer unit normal vector of $K$ at the boundary point $\rho_{K}(u)u$
belongs to Borel set $\eta$ on the unit sphere $S^{n-1}$.

Furthermore, Lutwak, Yang and Zhang \cite{LYZ2018} also introduced the $L_p$ dual curvature measures
of which the $L_p$ surface area measures, the $L_p$ integral curvature measures and the dual curvature measures
are special cases.

For $p,q\in \mathbb{R}$, the $L_p$ dual curvature measure $\widetilde{C}_{p,q}(K,Q,\cdot)$ of $K\in \mathcal{K}_{o}^n$ with respect to $Q\in \mathcal{S}_{o}^n$
is given  by
\begin{align}
d\widetilde{C}_{p,q}(K,Q,\cdot)=h_K^{-p}d\widetilde{C}_{q}(K,Q,\cdot),
\end{align}
where $h_K(u)=\max\{u\cdot x: x\in K\}$ for $u\in S^{n-1}$ is
the support function of a convex body $K\in \mathcal{K}_{o}^n$.
The $L_p$ dual Minkowski problem for the $L_p$ dual curvature measure is stated as follows:
\vskip 0.2cm
\noindent
{\bf The $L_p$ dual Minkowski problem} \cite{LYZ2018}:
\emph{Suppose that $p,q\in \mathbb{R}$, and $Q\in \mathcal{S}_{o}^n$ are fixed.
Given a finite Borel measure $\mu$ on the unit sphere $S^{n-1}$, find necessary and sufficient condition(s) on $\mu$ so that
it becomes the $L_p$ dual curvature measure $\widetilde{C}_{p,q}(K,Q,\cdot)$ of a convex body $K\in \mathcal{K}_{o}^n$.}
\vskip 0.2cm

When the given measure $\mu$ has a density $f$, the $L_p$ dual Minkowski problem becomes the following
Monge-Amp\`{e}re type equation on $S^{n-1}$:
\begin{align}\label{LpdualMPequation}
\text{det}(\nabla^2 h(u)+h(u)I)=h^{p-1}(u)\|\nabla h(u)+h(u)u\|_Q^{n-q}f(u),
\end{align}
where $f$ is the given ``data" function on $S^{n-1}$, $h$ is the function to be found, and $\|\cdot\|_Q$ is the Minkowski functional of $Q$ defined by $\|x\|_Q=\inf\{\lambda>0, x\in \lambda Q\}$. Here, $\nabla h(u)$ and $\nabla^2 h(u)$ denote the gradient vector and the Hessian matrix of $h$, respectively, with respect to a moving orthonormal frame on $S^{n-1}$, and $I$ is the identity matrix.

The case for $q=n$ is the $L_p$ Minkowski problem.
The case for $q=0$ and $Q=B $ is the $L_p$ Aleksandrov problem \cite{Aleksandrov,HuangLYZ2018}.
The case for $p=0$ and $Q=B$ is the dual Minkowski problem \cite{HuangLYZ}.
In \cite{HuangLYZ}, the existence of solutions for the dual Minkowski problem for even data with in the class of origin-symmetric convex bodies was established. The existence for the critical cases of the even dual Minkowski problem were established in \cite{Zhao2} and \cite{BoroczkyHP}, and a complete solution to the dual Minkowski problem with negative indices was given in \cite{Zhao1}. The continuity of the solution to the dual Minkowski problem for negative indices was obtained in \cite{WFZ}.

Recently, the $L_p$ dual Minkowski problem receives much attention.
Huang and Zhao \cite{HuangZ2018} gave a complete characterization to existence parts of the $L_p$ dual Minkowski problem for
$q<0<p$ and the even case for $p,q>0$ with $p\neq q$.
Moreover, the
existence and uniqueness of smooth solution of \eqref{LpdualMPequation} were obtained by using the
method of continuity for prescribed smooth function $f$ and $p>q$ in \cite{HuangZ2018}.
B\"{o}r\"{o}czky and  Fodor \cite{BoroczkyF} obtained the existence part of the $L_p$ dual Minkowski problem for $p>1$ and $q>0$ with $p\neq q$. Huang, Lutwak, Yang and Zhang \cite{HuangLYZ2018} posed the $L_p$ Aleksandrov problem and established the existence of this problem in several situations for $p$. Zhao \cite{Zhao3} gave a necessary and sufficient condition for the existence part of the
even discrete $L_p$ Aleksandrov problem for $-1<p<0$.
Chen, Huang and Zhao \cite{ChenHZ}
obtained the existence of smooth solutions of the $L_p$ dual Minkowski problem \eqref{LpdualMPequation}
when $pq\geq 0$ and $f$ is even.
Li, Liu and Lu \cite{LiLL} obatined nonuniqueness of the solution to the $L_p$ dual Minkowski problem for $p<0<q$ by constructing an example.

In \cite{LYZ2018}, Lutwak, Yang and Zhang established the uniqueness of the solution to the $L_p$ dual Minkowski problem
for the case of polytopes when $q<p$:

\begin{theoremalph}[\cite{LYZ2018}]\label{uniquepolytope}
Let $q<p$ and $Q\in\mathcal{S}_{o}^n$.
If $P,P'\in \mathcal{K}_{o}^n$ be polytopes and
\begin{align*}
\widetilde{C}_{p,q}(P,Q,\cdot)=\widetilde{C}_{p,q}(P',Q,\cdot),
\end{align*}
then $P=P'$.
\end{theoremalph}

In this paper, motivated by works of Lutwak, Yang and Zhang \cite{LYZ2018}, we consider the uniqueness and continuity of the solution to the $L_p$ dual Minkowski problem.
Firstly, we extend Theorem \ref{uniquepolytope} to the case for general convex bodies (not necessarily polytopes). To achieve this goal,
we establish some new Minkowski-type inequalities ( see \eqref{MIp}, \eqref{MIp=0} and \eqref{MIq=0})
which are closely related to the optimization problem associated with the $L_p$ dual Minkowski problem.
One of our main results is as follows:

\begin{thm}\label{uniqueness}
Let $q<p$ and $Q\in \mathcal{S}_{o}^n$. If $K,L\in\mathcal{K}_o^n$
and
\begin{align*}
\widetilde{C}_{p,q}(K,Q,\cdot)=\widetilde{C}_{p,q}(L,Q,\cdot),
\end{align*}
then $K=L$.
\end{thm}

When $p=0$ and $Q=B$, Theorem \ref{uniqueness} is the uniqueness of the solution to the dual Minkowski problem for negative indices established by Zhao \cite{Zhao1}, but our method is completely different from Zhao's method in \cite{Zhao1}.
Thus, we give a new proof of the uniqueness of the solutin to the dual Minkowski problem for negative indices.

When $q=0$ and $Q=B$, the uniqueness of the solution to the $L_p$ Aleksandrov problem is obtained for $p>0$ by Theorem \ref{uniqueness}.
Besides, Huang, Lutwak, Yang and Zhang \cite{HuangLYZ2018} gave a complete solution to the existence of this problem.
Hence, the uniqueness and existence of the $L_p$ Aleksandrov problem for $p>0$ are solved completely.

Then, we consider the continuity of the solution to the $L_p$ dual Minkowski problem:
\vskip 0.2cm
\emph{Let  $q<p$, $Q\in \mathcal{S}_{o}^n$ and $K_i\in \mathcal{K}^n_o$ for each $i=0, 1, 2, \cdots$.
 Does $\{K_i\}$ converge to $K$ in Hausdorff metric as $\{\widetilde{C}_{p,q}(K_i,\cdot)\}$ converges weakly to $\widetilde{C}_{p,q}(K_0,\cdot)$?}
\vskip 0.2cm

By Theorem \ref{uniqueness} and these new Minkowski-type inequlities (\eqref{MIp}, \eqref{MIp=0} and \eqref{MIq=0}), we obtain the following theorems for continuity.

\begin{thm}\label{thm1-1}
Let $q<0\leq p$, $Q\in\mathcal{S}_{o}^n$ and $K_i\in\mathcal{K}_{o}^n$ for each $i=0,1,2,\cdots$.
If $\{\widetilde{C}_{p,q}(K_i,Q,\cdot)\}$ converges to $\widetilde{C}_{p,q}(K_0,Q,\cdot)$ weakly,
then $\{K_i\}$ converges to $K_0$ in the Hausdorff metric.
\end{thm}

\begin{thm}\label{thm2-1}
Let $p\geq 1$ and $0\leq q<p$, $Q\in\mathcal{S}_{o}^n$,
and $K_i\in\mathcal{K}_{o}^n$ for each $i=0,1,2,\cdots$.
If $\{\widetilde{C}_{p,q}(K_i,Q,\cdot)\}$ converges to $\widetilde{C}_{p,q}(K_0,Q,\cdot)$ weakly,
then $\{K_i\}$ converges to $K_0$ in the Hausdorff metric.
\end{thm}

%%%%%%%%%%%%%%%%%%%%%%%%%%%%%%%%%%%%%%%%%%%%%%%%%%%%%%%%%%%%%%%%%%%%%%%%%%%%%
\vskip 0.5cm
\section{Preliminaries}\label{Preliminaries}
\vskip 0.3cm
%%%%%%%%%%%%%%%%%%%%%%%%%%%%%%%%%%%%%%%%%%%%%%%%%%%%%%%%%%%%%%%%%%%%%%%%%%%%%%

%In this section, we list some notations and recall some basic facts about convex bodies.
%More detail general references for the theory of convex bodies are based on
% Gardner \cite{Gardner2}, Gruber \cite{Gruber} and Schneider \cite{Schneider}.

The sets appearing in this paper are subsets in the $n$-dimensional Euclidean space $\mathbb{R}^n$.
The standard inner product of the vectors $x,y\in\mathbb{R}^n$ is denoted by $x\cdot y$.
Let $|x|=\sqrt{x\cdot x}$ be the Euclidean norm of $x\in\mathbb{R}^n$.
We write
$S^{n-1}$ for the boundary of the Euclidean unit
ball $B=\{x\in\mathbb{R}^n: |x|\leq 1\}$ in $\mathbb{R}^n$. The volume of $B$ is denoted by $\omega_n$.
The boundary and the set of all interiors of the subset $K$ of $\mathbb{R}^n$ are denoted by $\partial K$ and $\text{int}K$, respectively. Write $V(K)$ for the volume of a convex body $K$ in $\mathbb{R}^n$.
Let $C^+(S^{n-1})$ be the set of positive continuous functions on $S^{n-1}$.
All basic concepts and fundamental notations can refer to \cite{HuangLYZ,LYZ2018,Schneider}.

The support function $h_K: \mathbb{R}^{n}\rightarrow \mathbb{R}$ of a compact convex set
$K$ in $\mathbb{R}^n$ is defined by
\begin{align}\label{supportf}
h_K(x)=\max\{x\cdot y: y\in K\},\quad x\in \mathbb{R}^{n}.
\end{align}
Note that the support function is positively homogeneous of degree 1 and is sublinear.
It is clear that $h_K\in C^+(S^{n-1})$ for $K\in \mathcal{K}_{o}^n$.
The support hyperplane $H_K$ of $K\in \mathcal{K}_{o}^n$ with respect to
outer unit normal $u\in S^{n-1}$ is defined by
\begin{align*}
H_K(u)=\{x\in\mathbb{R}^n: x\cdot u=h_K(u)\}.
\end{align*}

Let $K\subseteq\mathbb{R}^n$ be a compact star shaped set with respect to the origin, then
its radial function
$\rho_K: \mathbb{R}^{n}\setminus \{0\}\rightarrow \mathbb{R}$ is given by
\begin{align*}
\rho_K(x)=\max\{\lambda\geq0: \lambda x\in K\},
\end{align*}
for each $x\in \mathbb{R}^{n}\setminus \{0\}$.
Note that the radial function is positively homogeneous of degree -1.
It is easy to see that the radial function $\rho_K\in C^+(S^{n-1})$
for $K\in \mathcal{S}_{o}^n$.
Moreover, for $K\in \mathcal{S}_{o}^n$,
\begin{align}\label{starbboundary}
\partial K=\{\rho_K(u)u: u\in S^{n-1}\}=\{\rho_K(x)x: x\in \mathbb{R}^{n}\setminus \{0\}\}
=\{x\in \mathbb{R}^{n}: \rho_K(x)=1\}.
\end{align}

It is known that the set $\mathcal{K}_{o}^n$ can be endowed with the following two metrics: The first one is the Hausdorff metric, the distance between $K, L\in \mathcal{K}_{o}^n$,
\begin{align*}
|| h_K-h_L ||=\mathop{\max}\limits_{u\in S^{n-1}}|h_K(u)-h_L(u)|.
\end{align*}
The second metric is the radial metric, the distance between $K, L\in \mathcal{K}_{o}^n$,
\begin{align*}
|| \rho_K-\rho_L ||=\mathop{\max}\limits_{u\in S^{n-1}}|\rho_K(u)-\rho_L(u)|.
\end{align*}
Note that these two metrics are mutually equivalent, that is, if $K, K_i\in \mathcal{K}_{o}^n$, then
\begin{align*}
h_{K_i}\rightarrow h_K~\text{uniformly}\quad \text{if and only if} \quad  \rho_{K_i}\rightarrow \rho_K~\text{uniformly}.
\end{align*}
Hence, we may write $\{K_{i}\}$ converges to $K$ without specifying which metric is in use.

For each $K\in \mathcal{K}_{o}^n$,  we use $K^*$ to denote the polar body of $K$:
\begin{align*}
K^*=\{ x\in\mathbb{R}^n: x\cdot y\leq1~\text{for~all}~y\in K\}.
\end{align*}
It is clear that $K^*\in \mathcal{K}_{o}^n$ and $K=(K^{*})^*$ for all $K\in \mathcal{K}_{o}^n$. By this definition, we know that an important fact between $K$ and $K^*$ on $\mathbb{R}^{n}\setminus \{0\}$ is
\begin{align}\label{supportradialf}
h_K=1/\rho_{K^*}\quad \text{and}\quad \rho_K=1/h_{K^*}.
\end{align}
If $K,K_i\in \mathcal{K}_{o}^n$, then
\begin{align}\label{polareq}
K_i\rightarrow K\quad \text{if and only if} \quad  K^*_i\rightarrow K^*.
\end{align}

For each $h\in C^+(S^{n-1})$, the Wulff shape determined by $h$, denoted $[h]$, is the convex body given by
\begin{align}\label{Wulffs}
[h]=\{x\in\mathbb{R}^n: x\cdot v\leq h(v)~\text{for~all}~v\in S^{n-1}\}.
\end{align}
It is clear that $[h]\in \mathcal{K}_{o}^n$ and
\begin{align*}
h_{[h]}\leq h,
\end{align*}
and if $K\in\mathcal{K}_{o}^n$, then
\begin{align*}
[h_K]=K.
\end{align*}

For given $h_0\in C^+(S^{n-1})$, $f\in C(S^{n-1})$, and small enough $\delta>0$, the continuous function $h_t:S^{n-1}\rightarrow(0,+\infty)$ is defined for each $t\in(-\delta, \delta)$ by
\begin{align}\label{logfws}
\log h_t(v)=\log h_0(v)+tf(v)+o(t,v),\quad v\in S^{n-1},
\end{align}
where $o(t,\cdot)\in C(S^{n-1})$ and $\lim_{t\rightarrow0}o(t,\cdot)/t=0$ uniformly on $S^{n-1}$.
$[h_t]$ is called a logarithmic family of Wulff shapes generated by $(h_0, f)$.
If $h_0$ is the support function $h_K$ of a convex body $K\in\mathcal{K}_{o}^n$,
$[h_t]$ is written as $[K,f,t]$.

If $K$, $L$ are convex bodies in $\mathbb{R}^n$ and $s,t\geq 0$, the Minkowski combination $sK+tL$ is
%a convex body in $\mathbb{R}^n$
defined by
\begin{align*}\label{}
sK+tL=\{sx+ty: x\in K, y\in L\},
\end{align*}
or equivalently,
\begin{align*}
h_{sK+tL}=sh_K+th_L.
\end{align*}

For $K,L\in \mathcal{K}_{o}^n$ and $s,t\geq 0$, the $L_p$ Minkowski combination $s\cdot K+_pt\cdot L$ for $p\geq 1$ is the compact convex set
defined by
\begin{align*}
h_{s\cdot K+_pt\cdot L}^p=sh_K^p+th_L^p.
\end{align*}
If $sh_K^p+th_L^p>0$ on $S^{n-1}$, then $s\cdot K+_pt\cdot L$ $(\in \mathcal{K}_{o}^n)$ can
be written by the Wulff shape \eqref{Wulffs} in the form
\begin{align*}
s\cdot K+_pt\cdot L=[(sh_K^p+th_L^p)^\frac{1}{p}].
\end{align*}

By concept of Wulff shape,
the definition of
an $L_p$ Minkowski combination $s\cdot  K+_pt\cdot L$
can be extended to $p<1$ and even negative $s$ or $t$:

Suppose $K,L\in \mathcal{K}_{o}^n$ and $s,t\in\mathbb{R}$ such that $sh_K^p+th_L^p>0$ on $S^{n-1}$.
The $L_p$ Minkowski combination $s\cdot  K+_pt\cdot L$ is given for $p\neq 0$ by
\begin{align}\label{LpMinkowskic1}
s\cdot  K+_pt\cdot L=\bigcap_{u\in S^{n-1}}\Big\{x: x\cdot u\leq \big(sh_{K}^p(u)+th_{L}^p(u)\big)^\frac{1}{p}\Big\}.
\end{align}
When $p=0$, the $L_0$ Minkowski combination $s\cdot  K+_0t\cdot L$ is defined by
\begin{align}\label{LpMinkowskic2}
s\cdot  K+_0t\cdot L=\bigcap_{u\in S^{n-1}}\big\{x: x\cdot u\leq h_{K}^{s}(u)h_{L}^t(u)\big\},
\end{align}
for $K,L\in \mathcal{K}_{o}^n$ and $s,t\in\mathbb{R}$.

It follows from $K,L\in \mathcal{K}_{o}^n$ that $s\cdot K+_pt\cdot L\in \mathcal{K}_{o}^n$ for all $p\in\mathbb{R}$.

For $K,L\in \mathcal{K}_{o}^n$ and $p\geq 1$,
the $L_p$ Brunn-Minkowski inequality \cite{Schneider} is
\begin{align}\label{LpBMI}
V(K+_pL)^{\frac{p}{n}}\geq V(K)^{\frac{p}{n}}+V(L)^{\frac{p}{n}}.
\end{align}
When $p=1$, the equality in \eqref{LpBMI} holds if and only if $K$ and $L$ are homothetic.
When $p>1$, the equality for some $p$ in \eqref{LpBMI} holds if and only if $K$ and $L$ are dilates.

Suppose $K,L\subseteq\mathbb{R}^n$ are compact star shaped sets with respect to the origin and $s,t\geq0$.
The radial combination $sK\tilde{+}tL$ is the compact star shaped set with respect to the origin given by
\begin{align*}
sK\tilde{+}tL=\{sx+ty: x\in K, y\in L~\text{and}~x\cdot y=|x||y|\}.
\end{align*}
Note that the radial combination implies
\begin{align*}
\rho_{sK\tilde{+}tL}=s\rho_K+t\rho_L.
\end{align*}
The radial $q$-combination $s\cdot K\tilde{+}_qt\cdot L$ for $K,L\in \mathcal{K}_{o}^n$ and $s,t\geq0$ is defined by
\begin{align}
\rho_{s\cdot K\tilde{+}_qt\cdot L}^q&=s\rho_K^q+t\rho_L^q,\quad q\neq0,\label{radialc1}\\
\rho_{s\cdot K\tilde{+}_0t\cdot L}&=\rho_K^s\rho_L^t.\label{radialc2}
\end{align}
In order to have a natural definition of $s\cdot K\tilde{+}_0t\cdot L$ whose radial function is homogeneous of degree $-1$, it is necessary in \eqref{radialc2} that $s+t=1$.

In \cite{Lutwak1}, Lutwak introduced the dual mixed volume:

For star bodies $K, L\in \mathcal{S}^n_o$ and real number $q\in\mathbb{R}$, the $q$-th dual mixed volume of $K$
and $L$ is defined by
\begin{align*}
\widetilde{V}_{q}(K,L)=\frac{1}{n}\int_{S^{n-1}}\rho^{q}_{K}(u)\rho^{n-q}_{L}(u)du.
\end{align*}
In particular, $\widetilde{V}_{0}(K,L)=V(L)$ and $\widetilde{V}_{n}(K,L)=V(K)$.
The dual quermassintegral $\widetilde{W}_q(K)$ is given by  $\widetilde{W}_q(K)=\widetilde{V}_{n-q}(K,B)$.

For $0<q<n$, the dual Minkowski inequality and the dual Brunn-Minkowski inequality \cite{Schneider} are, respectively,
\begin{align}
\widetilde{V}_q(K,L)^n&\leq V(K)^qV(L)^{n-q},\label{dualMI}\\
V(K\tilde{+}_qL)^{\frac{q}{n}}&\leq
V(K)^{\frac{q}{n}}+V(L)^{\frac{q}{n}}.\label{dualBMI}
\end{align}
The equality in each of the above inequalities holds if and only if $K$ and $L$ are dilates.

The dual mixed entropy $\widetilde{E}(K,L)$ of $K,L\in \mathcal{S}^n_o$ is  defined by
\begin{align*}
\widetilde{E}(K,L)=\frac{1}{n}\int_{S^{n-1}}\log\bigg(\frac{\rho_{K}(u)}{\rho_L(u)}\bigg)\rho_L(u)^ndu.
\end{align*}

For each subset $\omega\subseteq S^{n-1}$ and convex body $K\in \mathcal{K}_{o}^n$, the radial Gauss image $\bm{\alpha}_K(\omega)$ of $\omega$ is defined in \cite{HuangLYZ} by
\begin{align*}
\bm{\alpha}_K(\omega)=\{v\in S^{n-1}: \rho_K(u)u\in H_K(v)~\text{for~some}~u\in \omega\}.
\end{align*}
In particular, $\bm{\alpha}_K(\{u\})$ is abbreviated as $\bm{\alpha}_K(u)$.
Let $\omega_K$ denote the set which is made up of all $u\in S^{n-1}$ such
that the set $\bm{\alpha}_K(u)$ contains more than one single element.
It follows that $\omega_K$ has spherical Lebesgue measure $0$ from \cite{Schneider}.
Thus, the radial Gauss map $\alpha_K: S^{n-1}\backslash \omega_K\rightarrow S^{n-1}$ of $K$ is given by
$\rho_K(u)u\in H_K(\alpha_K(u))$ for $u\in S^{n-1}\backslash \omega_K$.

Similarly, for each $\eta\subseteq S^{n-1}$ and $K\in \mathcal{K}_{o}^n$, the reverse radial Gauss image $\bm{\alpha}^*_K(\eta)$ of $\eta$ is defined in \cite{HuangLYZ} by
\begin{align*}
\bm{\alpha}^*_K(\eta)=\{u\in S^{n-1}: \rho_K(u)u\in H_K(v)~\text{for~some}~v\in \eta\}.
\end{align*}
When $\eta=\{v\}$, we abbreviate $\bm{\alpha}^*_K(\{v\})$ by $\bm{\alpha}^*_K(v)$. Note that
\begin{align*}
u\in\bm{\alpha}^*_K(v)\quad \text{if and only if} \quad  v\in\bm{\alpha}_K(u).
\end{align*}

The $L_p$ dual curvature measures are constructed in \cite{LYZ2018} and have the following
integral representation:

For each Borel set $\eta\subseteq S^{n-1}$, $p,q\in \mathbb{R}$ and $Q\in \mathcal{S}_{o}^n$, the $L_p$ dual curvature measure $\widetilde{C}_{p,q}(K,Q,\cdot)$ of $K\in \mathcal{K}_{o}^n$
is given  by
\begin{align}\label{Lpdcmintegral}
\widetilde{C}_{p,q}(K,Q,\eta)=\frac{1}{n}\int_{\bm{\alpha}^*_K(\eta)}h_K(\alpha_K(u))^{-p}\rho^q_{K}(u)\rho^{n-q}_{Q}(u)du.
\end{align}
It is not hard to see that
\begin{align}\label{Lpdcmh}
\widetilde{C}_{p,q}(\lambda K,\xi Q,\cdot)=\lambda^{q-p}\xi^{n-q}\widetilde{C}_{p,q}(K,Q\cdot),\quad \lambda, \xi>0.
\end{align}

For $ K,L\in \mathcal{K}_{o}^n$ and $Q\in \mathcal{S}_{o}^n$, the $L_p$ dual mixed volume $\widetilde{V}_{p,q}(K,L,Q)$ is defined by
\begin{align}\label{Lpdualmv}
\widetilde{V}_{p,q}(K,L,Q)=\int_{S^{n-1}}h_L^p(u)d\widetilde{C}_{p,q}(K, Q,u).
\end{align}
It is clear that $\widetilde{C}_{p,q}(K, Q,S^{n-1})=\widetilde{V}_{p,q}(K,B,Q)$. The $L_p$ mixed volume and the dual mixed volume
are the special cases, that is,
\begin{align}
\widetilde{V}_{p,q}(K,L,K)&=V_p(K,L),\label{Lpduamv-mv}\\
\widetilde{V}_{p,q}(K,K,Q)&=\widetilde{V}_q(K,Q).\label{Lpdualmvdmv}
\end{align}

The $L_p$ dual curvature measure $\widetilde{C}_{p,q}(K,Q,\cdot)$ can be represented by
the $q$-th dual curvature measure $\widetilde{C}_{q}(K,Q,\cdot)$ as follows:
\begin{align}\label{Lpdualcm2}
d\widetilde{C}_{p,q}(K,Q,\cdot)=h_K^{-p}d\widetilde{C}_{q}(K,Q,\cdot),
\end{align}
where the $q$-th dual curvature measure $\widetilde{C}_{q}(K,Q,\cdot)$ of $K\in \mathcal{K}_{o}^n$ with respect to $Q\in \mathcal{S}_{o}^n$ is given by
\begin{align*}
\widetilde{C}_{q}(K,Q,\eta)=\frac{1}{n}\int_{\bm{\alpha}^*_K(\eta)}\rho^q_{K}(u)\rho^{n-q}_{Q}(u)du.
\end{align*}
It is clear that the $q$-th dual curvature measure is the case $p=0$ of the $L_p$ dual curvature measure, that is,
\begin{align}\label{Lpdualcmp=0}
\widetilde{C}_{q}(K, Q,\cdot)=\widetilde{C}_{0,q}(K, Q,\cdot).
\end{align}
Note that $\widetilde{C}_{q}(K, Q,S^{n-1})=\widetilde{V}_q(K,Q)$.

%%%%%%%%%%%%%%%%%%%%%%%%%%%%%%%%%%%%%%%%%%%%%%%%%%%%%%%%%%%%%%%%%%%%%%%%%%%%%
\vskip 0.5cm
\section{Uniqueness}\label{}
\vskip 0.3cm
%%%%%%%%%%%%%%%%%%%%%%%%%%%%%%%%%%%%%%%%%%%%%%%%%%%%%%%%%%%%%%%%%%%%%%%%%%%%%%

In this section, we will consider
the uniqueness of the solution to the $L_p$ dual Minkowski problem for the case of convex bodies
by constructing some new Minkowski style inequalities.
The connection between the $L_p$ Minkowski combination (see \eqref{LpMinkowskic1} and \eqref{LpMinkowskic2}) and the radial $p$-combination (see \eqref{radialc1} and \eqref{radialc2}) is established as follows:
\begin{lem}\label{inclusion}
Let $K,L\in\mathcal{K}_o^n$, $p\in\mathbb{R}$ and $t\in[0,1]$, then
\begin{align}\label{inclusion-0}
(1-t)\cdot K\tilde{+}_pt\cdot L\subseteq(1-t)\cdot  K+_pt\cdot L.
\end{align}
The equality for any $t\in(0,1)$ holds in \eqref{inclusion-0} if and only if $K$ and $L$ are dilates when $p\in[1,n)$.
Moreover,
the equality for all $t\in(0,1)$ holds in \eqref{inclusion-0} if $K$ and $L$ are dilates when $p\in\mathbb{R}\setminus [1,n)$.
\end{lem}

\begin{proof}
Since $K,L\in\mathcal{K}_o^n$ and $t\in[0,1]$, by the definitions of the radial $p$-combination and the $L_p$ Minkowski combination, we have
\begin{align}\label{inclusion-01}
(1-t)\cdot K\tilde{+}_pt\cdot L \in \mathcal{S}_o^n\quad\text{and}\quad
(1-t)\cdot  K+_pt\cdot L\in \mathcal{K}_o^n,
\end{align}
for all $p\in\mathbb{R}$.
By \eqref{LpMinkowskic1},
the $L_p$ Minkowski combination $(1-t)\cdot  K+_pt\cdot L$ of $K$ and $L$ is given for $p\neq 0$ by
\begin{align}\label{inclusion-1}
(1-t)\cdot  K+_pt\cdot L=\bigcap_{u\in S^{n-1}}\Big\{x: x\cdot u\leq \big((1-t)h_{K}^p(u)+th_{L}^p(u)\big)^\frac{1}{p}\Big\}.
\end{align}
When $p=0$, by \eqref{LpMinkowskic2},
\begin{align*}
(1-t)\cdot  K+_0t\cdot L=\bigcap_{u\in S^{n-1}}\big\{x: x\cdot u\leq h_{K}^{1-t}(u)h_{L}^t(u)\big\}.
\end{align*}

When $p\neq 0$, given a unit vector $v\in S^{n-1}$,  we obtain
\begin{align*}
u\cdot \big((1-t)\rho_{K}^p(v)+t\rho_{L}^p(v)\big)^\frac{1}{p}v
&=\big((1-t)\rho_{K}^p(v)+t\rho_{L}^p(v)\big)^\frac{1}{p}u\cdot v\\
&=\Big((1-t)\big(u\cdot\rho_{K}(v)v\big)^p +t\big(u\cdot\rho_{L}(v)v\big)^p \Big)^\frac{1}{p}\\
&\leq \big((1-t)h_{K}^p(u)+th_{L}^p(u)\big)^\frac{1}{p},
\end{align*}
for all $u\in\{u'\in S^{n-1}:u'\cdot v>0\}$.
Since $K,L\in\mathcal{K}_o^n$, then
\begin{align*}
u\cdot \big((1-t)\rho_{K}^p(v)+t\rho_{L}^p(v)\big)^\frac{1}{p}v
\leq 0< \big((1-t)h_{K}^p(u)+th_{L}^p(u)\big)^\frac{1}{p},
\end{align*}
for all $u\in\{u'\in S^{n-1}:u'\cdot v\leq0\}$.
Hence, When $p\neq 0$, given a unit vector $v\in S^{n-1}$, we have
\begin{align*}
u\cdot \big((1-t)\rho_{K}^p(v)+t\rho_{L}^p(v)\big)^\frac{1}{p}v
\leq \big((1-t)h_{K}^p(u)+th_{L}^p(u)\big)^\frac{1}{p},
\end{align*}
for all $u\in S^{n-1}$, that is,
\begin{align*}
\big((1-t)\rho_{K}^p(v)+t\rho_{L}^p(v)\big)^\frac{1}{p}v
\in\bigcap_{u\in S^{n-1}}\Big\{x: x\cdot u\leq \big((1-t)h_{K}^p(u)+th_{L}^p(u)\big)^\frac{1}{p}\Big\}.
\end{align*}
Together with \eqref{radialc1} and \eqref{inclusion-1}, for $p\neq 0$,
\begin{align*}
\rho_{(1-t)\cdot K\tilde{+}_pt\cdot L}(v)v=
\big((1-t)\rho_{K}^p(v)+t\rho_{L}^p(v)\big)^\frac{1}{p}v\in (1-t)\cdot  K+_pt\cdot L.
\end{align*}
Then, by \eqref{starbboundary},
\begin{align*}
\partial\big((1-t)\cdot K\tilde{+}_pt\cdot L\big)\subseteq(1-t)\cdot  K+_pt\cdot L.
\end{align*}
From \eqref{inclusion-01}, we have
\begin{align*}
(1-t)\cdot K\tilde{+}_pt\cdot L\subseteq(1-t)\cdot  K+_pt\cdot L,
\end{align*}
for $p\neq 0$.

By the similar way, we can prove the case $p=0$ of \eqref{inclusion-0}.
Thus, \eqref{inclusion-0} is valid for all $p\in\mathbb{R}$.

For $p\in[1,n)$,
combining \eqref{dualBMI} with \eqref{LpBMI}, we have
\begin{align*}
V\big((1-t)\cdot K\tilde{+}_pt\cdot L\big)^{\frac{p}{n}}\leq
(1-t)V(K)^{\frac{p}{n}}+tV(L)^{\frac{p}{n}}\leq V\big((1-t)\cdot  K+_pt\cdot L\big)^{\frac{p}{n}},
\end{align*}
that is,
\begin{align*}
V\big((1-t)\cdot K\tilde{+}_pt\cdot L\big)\leq V\big((1-t)\cdot  K+_pt\cdot L\big),
\end{align*}
with equality for any $t\in(0,1)$ if and only if $K$ and $L$ are dilates.
Hence, the equality for any $t\in(0,1)$ in \eqref{inclusion-0} holds if and only if $K$ and $L$ are dilates when $p\in[1,n)$.

When $p\in\mathbb{R}\setminus [1,n)$, if $K$ and $L$ are dilates, we obtain
the equality for all $t\in(0,1)$ in \eqref{inclusion-0} holds
by the definitions of the radial $p$-combination (see \eqref{radialc1} and \eqref{radialc2}) and the $L_p$ Minkowski combination (see \eqref{LpMinkowskic1} and \eqref{LpMinkowskic2}).
\end{proof}

The following lemmas are needed.

\begin{lem}
Let $Q\in \mathcal{S}_{o}^n$ and $t\in[0,1]$. If $K,L\in\mathcal{K}_o^n$, then, for $q<0<p$ or $0<q<p$,
\begin{align}\label{BMIp1}
\widetilde{V}_q\big((1-t)\cdot K+_pt\cdot L,Q\big)^\frac{p}{q}\geq(1-t) \widetilde{V}_q(K,Q)^\frac{p}{q}+t\widetilde{V}_q(L,Q)^\frac{p}{q},
\end{align}
and for $q<p<0$,
\begin{align}\label{BMIp2}
\widetilde{V}_q\big((1-t)\cdot K+_pt\cdot L,Q\big)^\frac{p}{q}\leq(1-t) \widetilde{V}_q(K,Q)^\frac{p}{q}+t\widetilde{V}_q(L,Q)^\frac{p}{q}.
\end{align}
The equality for any $t\in(0,1)$ in each of the inequalities holds if and only if $K$ and $L$ are dilates.
\end{lem}

\begin{proof}
By Lemma \ref{inclusion} and the Minkowski's inequality for integrals, for $q<0<p$ or $0<q<p$, we have
\begin{align}
&\widetilde{V}_q\big((1-t)\cdot K+_pt\cdot L,Q\big)^\frac{p}{q}\nonumber\\
&\quad=\left(\frac{1}{n}\int_{S^{n-1}}\rho^{q}_{(1-t)\cdot K+_pt\cdot L}(u)\rho^{n-q}_{Q}(u)du\right)^\frac{p}{q}\nonumber\\
&\quad\geq\left(\frac{1}{n}\int_{S^{n-1}}\rho^{q}_{(1-t)\cdot K\tilde{+}_pt\cdot L}(u)\rho^{n-q}_{Q}(u)du\right)^\frac{p}{q}\label{BMIp1-1}\\
&\quad=\left(\frac{1}{n}\int_{S^{n-1}}\big((1-t)\rho_{K}^p(u)+t\rho_{L}^p(u)\big)^\frac{q}{p}\rho^{n-q}_{Q}(u)du\right)^\frac{p}{q}\nonumber\\
&\quad\geq\left(\frac{1}{n}\int_{S^{n-1}}\big((1-t)\rho_{K}^p(u)\big)^\frac{q}{p}\rho^{n-q}_{Q}(u)du\right)^\frac{p}{q}
+\left(\frac{1}{n}\int_{S^{n-1}}\big(t\rho_{L}^p(u)\big)^\frac{q}{p}\rho^{n-q}_{Q}(u)du\right)^\frac{p}{q}\label{BMIp1-2}\\
&\quad=(1-t)\left(\frac{1}{n}\int_{S^{n-1}}\rho_{K}^{q}(u)\rho^{n-q}_{Q}(u)du\right)^\frac{p}{q}
+t\left(\frac{1}{n}\int_{S^{n-1}}\rho_{L}^{q}(u)\rho^{n-q}_{Q}(u)du\right)^\frac{p}{q}\nonumber\\
&\quad=(1-t)\widetilde{V}_q(K,Q)^\frac{p}{q}+t\widetilde{V}_q(L,Q)^\frac{p}{q}.\nonumber
\end{align}

For any $t\in(0,1)$,
if the equality holds in \eqref{BMIp1}, then the equality holds in \eqref{BMIp1-2}, that is, we have $K$ and $L$ are dilates
by the equality condition of the Minkowski's inequality for integrals.
If $K$ and $L$ are dilates, then the both equalities in \eqref{BMIp1-1} and \eqref{BMIp1-2} hold by Lemma \ref{inclusion} and the Minkowski's inequality for integrals. That is, the equality holds in \eqref{BMIp1}.
Hence, the equality for any $t\in(0,1)$ holds in \eqref{BMIp1} if and only if $K$ and $L$ are dilates.

By Lemma \ref{inclusion} and the Minkowski's inequality for integrals again, we can prove \eqref{BMIp2} for $q<p<0$.
\end{proof}

\begin{remark}
When $q=p>0$, \eqref{BMIp1-2} is a identity. Hence, by Lemma \ref{inclusion}, \eqref{BMIp1} is valid
 and the equality holds in \eqref{BMIp1} if and only if $K$ and $L$ are dilates when $q=p\in[1,n)$.

When $q=p<0$, \eqref{BMIp2} is also valid by Lemma \ref{inclusion}.
\end{remark}

When $p=1$ in \eqref{BMIp1}, by $\widetilde{W}_q(K)=\widetilde{V}_{n-q}(K,B)$, we obtain the following Brunn-Minkowski inequality for dual quermassintegral:
\begin{cor}
If $K,L\in\mathcal{K}_o^n$ and $q\geq n-1$ with $q\neq n$, then
\begin{align*}\label{}
\widetilde{W}_q(K+L)^\frac{1}{n-q}
\geq \widetilde{W}_q(K)^\frac{1}{n-q}+\widetilde{W}_q(L)^\frac{1}{n-q},
\end{align*}
with equality if and only if $K$ and $L$ are dilates.
\end{cor}

The following variational formula is needed.

\begin{lem}[\cite{LYZ2018}]\label{variationf}
Suppose $K\in\mathcal{K}_{o}^n$, and $f:S^{n-1}\rightarrow\mathbb{R}$ is continuous. Then, for $Q\in\mathcal{S}_{o}^n$ and $q\neq0$,
\begin{align}\label{variationalf1}
\lim_{t\rightarrow0}\frac{\widetilde{V}_q([K,f,t],Q)-\widetilde{V}_q(K,Q)}{t}=q\int_{S^{n-1}}f(v)d\widetilde{C}_q(K,Q,v),
\end{align}
and
\begin{align}\label{variationalf2}
\lim_{t\rightarrow0}\frac{\widetilde{E}([K,f,t],Q)-\widetilde{E}(K,Q)}{t}=\int_{S^{n-1}}f(v)d\widetilde{C}_0(K,Q,v).
\end{align}
\end{lem}

By Lemma \ref{variationf}, the following results are obtained.

\begin{lem}
Suppose $p\neq 0$ and $q\neq 0$. If $Q\in \mathcal{S}_{o}^n$ and $K,L\in\mathcal{K}_o^n$, then
\begin{align}\label{variationalfLp}
\lim_{t\rightarrow0}\frac{\widetilde{V}_q\big((1-t)\cdot K+_pt\cdot L,Q\big)-\widetilde{V}_q(K,Q)}{t}
=\frac{q}{p}\Big(\widetilde{V}_{p,q}(K,L,Q)-\widetilde{V}_{q}(K,Q)\Big),
\end{align}
\begin{align}\label{variationalfp=0}
\lim_{t\rightarrow0}\frac{\widetilde{V}_q\big((1-t)\cdot K+_0t\cdot L,Q\big)-\widetilde{V}_q(K,Q)}{t}
= q\int_{S^{n-1}}\log\frac{h_L(v)}{h_K(v)}d\widetilde{C}_q(K,Q,v).
\end{align}
\begin{align}\label{variationalfq=0}
\lim_{t\rightarrow0}\frac{\widetilde{E}\big((1-t)\cdot K+_pt\cdot L,Q\big)-\widetilde{E}(K,Q)}{t}
=\frac{1}{p}\Big(\widetilde{V}_{p,0}(K,L,Q)-V(Q)\Big),
\end{align}
\end{lem}

\begin{proof}
For sufficiently small $t$, define $h_t$ by
\begin{align}
h_t^p&=(1-t)h_{K}^p+th_{L}^p=h_{K}^p+t(h_{L}^p-h_{K}^p),\quad p\neq0,\label{variationalfLp-1}\\
h_t&=h_K^{1-t}h_L^t=h_K\left(\frac{h_L}{h_K}\right)^t, \quad\quad\quad\quad\quad\quad\quad p=0.\label{variationalfLp-2}
\end{align}
By the $L_p$ Minkowski combination, the Wulff shape $[h_t]=(1-t)\cdot  K+_pt\cdot L$.
From \eqref{variationalfLp-1} and \eqref{variationalfLp-2}, it follows that, for sufficiently small $t$,
\begin{align*}
\log h_t&=\log h_K+\frac{t}{p}\frac{h_{L}^p-h_{K}^p}{h_{K}^p}+o(t,\cdot),\quad p\neq0,\\
\log h_t&=\log h_K+t\log \frac{h_L}{h_K},\quad\quad\quad\quad\quad p=0.
\end{align*}

Let $f=\frac{1}{p}\frac{h_{L}^p-h_{K}^p}{h_{K}^p}$ when $p\neq0$, and let $f=\log \frac{h_L}{h_K}$ when $p=0$.
For $p\neq0$, by Lemma \ref{variationf},
\eqref{Lpdualcm2}, \eqref{Lpdualmv} and \eqref{Lpdualmvdmv}, we have
\begin{align*}
\lim_{t\rightarrow0}\frac{\widetilde{V}_q\big((1-t)\cdot K+_pt\cdot L,Q\big)-\widetilde{V}_q(K,Q)}{t}
&=\frac{q}{p}\int_{S^{n-1}}\frac{h_{L}^p(v)-h_{K}^p(v)}{h_{K}^p(v)}d\widetilde{C}_q(K,Q,v)\\
&=\frac{q}{p}\int_{S^{n-1}}(h_{L}^p(v)-h_{K}^p(v))d\widetilde{C}_{p,q}(K,Q,v)\\
&=\frac{q}{p}\Big(\widetilde{V}_{p,q}(K,L,Q)-\widetilde{V}_{q}(K,Q)\Big),
\end{align*}
and
\begin{align*}
\lim_{t\rightarrow0}\frac{\widetilde{E}\big((1-t)\cdot K+_pt\cdot L,Q\big)-\widetilde{E}(K,Q)}{t}
&=\frac{1}{p}\int_{S^{n-1}}\frac{h_{L}^p(v)-h_{K}^p(v)}{h_{K}^p(v)}d\widetilde{C}_0(K,Q,v)\\
&=\frac{1}{p}\int_{S^{n-1}}(h_{L}^p(v)-h_{K}^p(v))d\widetilde{C}_{p,0}(K,Q,v)\\
&=\frac{1}{p}\Big(\widetilde{V}_{p,0}(K,L,Q)-V(Q)\Big).
\end{align*}
When $p=0$, we obtain the desired formula \eqref{variationalfp=0} by \eqref{variationalf1}.
\end{proof}

The Minkowski style inequality for $q<p$ with $p\neq0$ and $q\neq0$ is established as follows:

\begin{thm}
Let $Q\in \mathcal{S}_{o}^n$. If $K,L\in\mathcal{K}_o^n$, then, for $q<p$ with $p\neq0$ and $q\neq0$
\begin{align}\label{MIp}
\left(\frac{\widetilde{V}_{p,q}(K,L,Q)}{\widetilde{V}_q(K,Q)}\right)^\frac{1}{p}
\geq\left(\frac{\widetilde{V}_{q}(L,Q)}{\widetilde{V}_q(K,Q)}\right)^\frac{1}{q},
\end{align}
with equality if and only if $K$ and $L$ are dilates.
\end{thm}

\begin{proof}
For $q<0<p$ or $0<q<p$,
by \eqref{BMIp1}, the function $f$ defined by
\begin{align*}
f(t)=\widetilde{V}_q\big((1-t)\cdot K+_pt\cdot L,Q\big)^\frac{p}{q}-(1-t)\widetilde{V}_q(K,Q)^\frac{p}{q}-t\widetilde{V}_q(L,Q)^\frac{p}{q}
\end{align*}
is nonnegative for $t\in[0,1]$ and satisfies $f(0)=f(1)=0$.
By \eqref{BMIp1} again, for $t_1,t_2\in [0,1]$ and $\lambda\in(0,1)$,
\begin{align*}
&f((1-\lambda) t_1+\lambda t_2)\\
&\quad=\widetilde{V}_q\Big(\big(1-((1-\lambda) t_1+\lambda t_2)\big)\cdot K+_p((1-\lambda) t_1+\lambda t_2)\cdot L,Q\Big)^\frac{p}{q}\\
&\quad\quad-\big(1-((1-\lambda) t_1+\lambda t_2)\big)\widetilde{V}_q(K,Q)^\frac{p}{q}
-((1-\lambda) t_1+\lambda t_2)\widetilde{V}_q(L,Q)^\frac{p}{q}\\
&\quad=\widetilde{V}_q\Big((1-\lambda)\cdot\big((1-t_1)\cdot K+_p t_1\cdot L\big)+_p\lambda \cdot\big((1-t_2)\cdot K+_pt_2\cdot L\big),Q\Big)^\frac{p}{q}\\
&\quad\quad-\big(1-((1-\lambda) t_1+\lambda t_2)\big)\widetilde{V}_q(K,Q)^\frac{p}{q}
-((1-\lambda) t_1+\lambda t_2)\widetilde{V}_q(L,Q)^\frac{p}{q}
\end{align*}
\begin{align*}
&\geq(1-\lambda)\widetilde{V}_q\big((1-t_1)\cdot K+_p t_1\cdot L,Q\big)^\frac{p}{q}+\lambda\widetilde{V}_q\big((1-t_2)\cdot K+_p t_2\cdot L,Q\big)^\frac{p}{q}\\
&\quad-\big(1-((1-\lambda) t_1+\lambda t_2)\big)\widetilde{V}_q(K,Q)^\frac{p}{q}
-((1-\lambda) t_1+\lambda t_2)\widetilde{V}_q(L,Q)^\frac{p}{q}\\
&=(1-\lambda)f(t_1)+\lambda f(t_2),
\end{align*}

and hence, $f(t)$ is concave for $t\in[0,1]$.
Together with \eqref{variationalfLp}, we have
\begin{align*}
\lim_{t\rightarrow0}\frac{f(t)-f(0)}{t}
&=\widetilde{V}_q(K,Q)^{\frac{p}{q}-1}\Big(\widetilde{V}_{p,q}(K,L,Q)-\widetilde{V}_q(K,Q)\Big)
+\widetilde{V}_q(K,Q)^{\frac{p}{q}}-\widetilde{V}_q(L,Q)^\frac{p}{q}\\
&=\widetilde{V}_q(K,Q)^{\frac{p}{q}-1}\widetilde{V}_{p,q}(K,L,Q)-\widetilde{V}_q(L,Q)^\frac{p}{q}\geq 0,
\end{align*}
that is,
\begin{align*}
\left(\frac{\widetilde{V}_{p,q}(K,L,Q)}{\widetilde{V}_q(K,Q)}\right)^\frac{1}{p}
\geq\left(\frac{\widetilde{V}_{q}(L,Q)}{\widetilde{V}_q(K,Q)}\right)^\frac{1}{q}.
\end{align*}
Moreover, $\lim_{t\rightarrow0}\frac{f(t)-f(0)}{t}=0$ only if $f(t)$ is identically 0 for $t\in[0,1]$. The latter implies the equality in \eqref{BMIp1}. Hence, the equality in \eqref{MIp} holds if and only if $K$ and $L$ are dilates.

Furthermore, we can prove \eqref{MIp} for $q<p<0$ by the arguments as in the proof of the case $q<0<p$ or $0<q<p$.

Therefore, we obtain the desired formula \eqref{MIp} for all $q<p$ with $p\neq0$ and $q\neq0$.
\end{proof}

If $Q=K$ in \eqref{MIp}, by $\widetilde{V}_{p,q}(K,L,K)=V_{p}(K,L)$ and $\widetilde{V}_q(K,K)=V(K)$, we obtain the following inequality:
\begin{cor}
If $K,L\in\mathcal{K}_o^n$, then, for $q<p$ with $p\neq0$ and $q\neq0$,
\begin{align*}\label{}
\left(\frac{V_{p}(K,L)}{V(K)}\right)^\frac{1}{p}
\geq\left(\frac{\widetilde{V}_{q}(L,K)}{V(K)}\right)^\frac{1}{q},
\end{align*}
with equality if and only if $K$ and $L$ are dilates.
\end{cor}

The following uniqueness of the solution to the $L_p$ dual Minkowski problem for $q<p$ with $p\neq0$ and $q\neq0$ is obained.

\begin{thm}\label{uniquenesspq0}
Let $Q\in \mathcal{S}_{o}^n$ and $q<p$ with $p\neq0$ and $q\neq0$. If $K,L\in\mathcal{K}_o^n$
and
\begin{align*}
\widetilde{C}_{p,q}(K,Q,\cdot)=\widetilde{C}_{p,q}(L,Q,\cdot),
\end{align*}
then $K=L$.
\end{thm}

\begin{proof}
By \eqref{Lpdualmvdmv}, \eqref{Lpdualmv} and $\widetilde{C}_{p,q}(K,Q,\cdot)=\widetilde{C}_{p,q}(L,Q,\cdot)$,
\begin{align*}
\widetilde{V}_q(K,Q)&=\widetilde{V}_{p,q}(K,K,Q)=\widetilde{V}_{p,q}(L,K,Q),\\
\widetilde{V}_q(L,Q)&=\widetilde{V}_{p,q}(L,L,Q)=\widetilde{V}_{p,q}(K,L,Q).
\end{align*}
Together with \eqref{MIp}, we have
\begin{align*}
\left(\frac{\widetilde{V}_q(K,Q)}{\widetilde{V}_q(L,Q)}\right)^\frac{1}{p}
=\left(\frac{\widetilde{V}_{p,q}(L,K,Q)}{\widetilde{V}_q(L,Q)}\right)^\frac{1}{p}
\geq\left(\frac{\widetilde{V}_{q}(K,Q)}{\widetilde{V}_q(L,Q)}\right)^\frac{1}{q},
\end{align*}
that is,
\begin{align}\label{uniqueness1-1}
\widetilde{V}_q(K,Q)^{\frac{1}{p}-\frac{1}{q}}\geq \widetilde{V}_q(L,Q)^{\frac{1}{p}-\frac{1}{q}},
\end{align}
with equality if and only if $K$ and $L$ are dilates. By \eqref{MIp} again,
\begin{align*}
\left(\frac{\widetilde{V}_q(L,Q)}{\widetilde{V}_q(K,Q)}\right)^\frac{1}{p}
=\left(\frac{\widetilde{V}_{p,q}(K,L,Q)}{\widetilde{V}_q(K,Q)}\right)^\frac{1}{p}
\geq\left(\frac{\widetilde{V}_{q}(L,Q)}{\widetilde{V}_q(K,Q)}\right)^\frac{1}{q},
\end{align*}
that is,
\begin{align}\label{uniqueness1-2}
\widetilde{V}_q(L,Q)^{\frac{1}{p}-\frac{1}{q}}\geq \widetilde{V}_q(K,Q)^{\frac{1}{p}-\frac{1}{q}},
\end{align}
with equality if and only if $K$ and $L$ are dilates.

Combining \eqref{uniqueness1-1} with \eqref{uniqueness1-2}, we obtain
$$\widetilde{V}_q(K,Q)=\widetilde{V}_q(L,Q),$$
and $K, L$ are dilates. It follows that $K=L$.
\end{proof}

Next, we consider the uniqueness of the solution to the $L_p$ dual Minkowski problem for $p=0$ and $q<0$. Firstly, the following log-Brunn-Minkowski style inequality is obtained.

\begin{lem}
Let $Q\in \mathcal{S}_{o}^n$ and $t\in[0,1]$. If $K,L\in\mathcal{K}_o^n$ and $q<0$, then
\begin{align}\label{BMIp=0}
\widetilde{V}_q\big((1-t)\cdot K+_0t\cdot L,Q\big)\leq \widetilde{V}_q(K,Q)^{1-t}\widetilde{V}_q(L,Q)^t,
\end{align}
with equality for any $t\in(0,1)$ if and only if $K$ and $L$ are dilates.
\end{lem}

\begin{proof}
For $t\in(0,1)$,
by Lemma \ref{inclusion} and the H\"{o}lder's inequality for integrals, we have
\begin{align}
&\widetilde{V}_q\big((1-t)\cdot K+_0t\cdot L,Q\big)\\
&\quad=\frac{1}{n}\int_{S^{n-1}}\rho^{q}_{(1-t)\cdot K+_0t\cdot L}(u)\rho^{n-q}_{Q}(u)du\nonumber\\
&\quad\leq\frac{1}{n}\int_{S^{n-1}}\rho^{q}_{(1-t)\cdot K\tilde{+}_0t\cdot L}(u)\rho^{n-q}_{Q}(u)du\label{BMIp=0-1}\\
&\quad=\frac{1}{n}\int_{S^{n-1}}(\rho_{K}^q(u))^{1-t}(\rho_{L}^q(u))^t\rho^{n-q}_{Q}(u)du\nonumber\\
&\quad\leq\left(\frac{1}{n}\int_{S^{n-1}}\rho_{K}^{q}(u)\rho^{n-q}_{Q}(u)du\right)^{1-t}
\left(\frac{1}{n}\int_{S^{n-1}}\rho_{L}^{q}(u)\rho^{n-q}_{Q}(u)du\right)^t\label{BMIp=0-2}\\
&\quad=\widetilde{V}_q(K,Q)^{1-t}\widetilde{V}_q(L,Q)^t.\nonumber
\end{align}

For any $t\in(0,1)$, if the equality in \eqref{BMIp=0} holds, then the equality in \eqref{BMIp=0-2} holds, that is, we have $K$ and $L$ are dilates by the equality condition of the H\"{o}lder's inequality for integrals.
If $K$ and $L$ are dilates, then the both equalities in \eqref{BMIp=0-1} and
\eqref{BMIp=0-2} hold by Lemma \ref{inclusion} and the H\"{o}lder's inequality for integrals.
Therefore,
the equality for any $t\in(0,1)$ holds in \eqref{BMIp=0} if and only if $K$ and $L$ are dilates.
\end{proof}

The log-Minkowski style inequality for $q<0$ is established as follows:

\begin{thm}
Let $Q\in \mathcal{S}_{o}^n$ and $q<0$. If $K,L\in\mathcal{K}_o^n$, then
\begin{align}\label{MIp=0}
\frac{1}{\widetilde{V}_q(K,Q)}\int_{S^{n-1}}\log\frac{h_{L}(v)}{h_K(v)}d\widetilde{C}_q(K,Q,v)
\geq\frac{1}{q}\log\frac{\widetilde{V}_q(L,Q)}{\widetilde{V}_q(K,Q)},
\end{align}
with equality if and only if $K$ and $L$ are dilates.
\end{thm}

\begin{proof}
By \eqref{BMIp=0}, the function $g$ defined by
\begin{align*}
g(t)=\log\widetilde{V}_q\big((1-t)\cdot K+_0t\cdot L,Q\big)-(1-t)\log\widetilde{V}_q(K,Q)-t\log\widetilde{V}_q(L,Q)
\end{align*}
is nonpositive for $t\in[0,1]$ and satisfies $g(0)=g(1)=0$. Moreover, by \eqref{BMIp=0} again, we obtain
\begin{align*}
g((1-\lambda) t_1+\lambda t_2)\leq (1-\lambda)g(t_1)+\lambda g(t_2),
\end{align*}
for $t_1,t_2\in [0,1]$ and $\lambda\in(0,1)$.
Hence, $g(t)$ is convex for $t\in[0,1]$.
Together with \eqref{variationalfp=0}, we have
\begin{align*}
\lim_{t\rightarrow0}\frac{g(t)-g(0)}{t}
&=\frac{q}{\widetilde{V}_q(K,Q)}\int_{S^{n-1}}\log\frac{h_{L}(v)}{h_K(v)}d\widetilde{C}_q(K,Q,v)
+\log\widetilde{V}_q(K,Q)-\log\widetilde{V}_q(L,Q)\\
&=\frac{q}{\widetilde{V}_q(K,Q)}\int_{S^{n-1}}\log\frac{h_{L}(v)}{h_K(v)}d\widetilde{C}_q(K,Q,v)
-\log\frac{\widetilde{V}_q(L,Q)}{\widetilde{V}_q(K,Q)}\leq 0,
\end{align*}
that is, the inequality \eqref{MIp=0} is established.

Furthermore, $\lim_{t\rightarrow0}\frac{g(t)-g(0)}{t}=0$ only if $g(t)$ is identically 0 for $t\in[0,1]$. The latter implies the equality in \eqref{BMIp=0}. Hence, the equality in \eqref{MIp=0} holds if and only if $K$ and $L$ are dilates.
\end{proof}

\begin{remark}
When $Q=B$, \eqref{MIp=0} is obtained in \cite{WFZ}, but their methods are different. In \cite{WFZ}, Wang, Fang and Zhou mainly applied the existence and uniqueness of dual Minkowski problem to prove \eqref{MIp=0} for $Q=B$. In this paper, we prove \eqref{MIp=0} by the log-Brunn-Minkowski style inequality \eqref{BMIp=0} and the variational formula \eqref{variationalfp=0}.
\end{remark}

The following theorem establishes the uniqueness of the solution to the $L_p$ dual Minkowski problem for $p=0$ and $q<0$.

\begin{thm}\label{uniquenessp=0}
Let $Q\in \mathcal{S}_{o}^n$ and $q<0$. If $K,L\in\mathcal{K}_o^n$
and
\begin{align*}
\widetilde{C}_{q}(K,Q,\cdot)=\widetilde{C}_{q}(L,Q,\cdot),
\end{align*}
then $K=L$.
\end{thm}

\begin{proof}
Since $\widetilde{C}_{q}(K,Q,\cdot)=\widetilde{C}_{q}(L,Q,\cdot)$, then
\begin{align}\label{uniquenessp=0-1}
\widetilde{V}_q(K,Q)=\widetilde{C}_{q}(K,Q,S^{n-1})
=\widetilde{C}_{q}(L,Q,S^{n-1})=\widetilde{V}_q(L,Q),
\end{align}
and
\begin{align}\label{uniquenessp=0-2}
\int_{S^{n-1}}\log\frac{h_{L}(v)}{h_K(v)}d\widetilde{C}_q(K,Q,v)=\int_{S^{n-1}}\log\frac{h_{L}(v)}{h_K(v)}d\widetilde{C}_q(L,Q,v).
\end{align}

Combining \eqref{MIp=0} with \eqref{uniquenessp=0-1}, we have
\begin{align}\label{uniquenessp=0-3}
\int_{S^{n-1}}\log\frac{h_{L}(v)}{h_K(v)}d\widetilde{C}_q(K,Q,v)\geq 0,
\end{align}
with equality if and only if $K$ and $L$ are dilates.
By \eqref{MIp=0} and \eqref{uniquenessp=0-1} again,
\begin{align*}
\int_{S^{n-1}}\log\frac{h_{K}(v)}{h_L(v)}d\widetilde{C}_q(L,Q,v)\geq 0,
\end{align*}
with equality if and only if $K$ and $L$ are dilates. Together with \eqref{uniquenessp=0-2}, we obtain
\begin{align}\label{uniquenessp=0-4}
\int_{S^{n-1}}\log\frac{h_{L}(v)}{h_K(v)}d\widetilde{C}_q(K,Q,v)
=\int_{S^{n-1}}\log\frac{h_{L}(v)}{h_K(v)}d\widetilde{C}_q(L,Q,v)
\leq 0,
\end{align}
with equality if and only if $K$ and $L$ are dilates.

From \eqref{uniquenessp=0-3} and \eqref{uniquenessp=0-4},
\begin{align*}
\int_{S^{n-1}}\log\frac{h_{L}(v)}{h_K(v)}d\widetilde{C}_q(K,Q,v)=0.
\end{align*}
It follows that $K$ and $L$ are dilates from the equality condition of \eqref{uniquenessp=0-3}.
Therefore, by \eqref{uniquenessp=0-1}, we obtain $K=L$.
\end{proof}

\begin{remark}
The case $Q=B$ of Theorem \ref{uniquenessp=0} is the uniqueness of the solutin to the dual Minkowski problem for negative indices established by Zhao \cite{Zhao1}. Our method is different from Zhao's method.
Thus, we give a new proof of the uniqueness of the solutin to the dual Minkowski problem for negative indices.
\end{remark}

Next, we will consider the uniqueness of the solution to the $L_p$ dual Minkowski problem for $p>0$ and $q=0$.

\begin{lem}
Let $Q\in \mathcal{S}_{o}^n$ and $t\in[0,1]$. If $K,L\in\mathcal{K}_o^n$ and $p>0$, then
\begin{align}\label{BMIq=01}
\exp\left(\frac{\widetilde{E}\big((1-t)\cdot K+_pt\cdot L,Q\big)}{V(Q)}\right)
\geq \left(\exp\Big(\frac{\widetilde{E}(K,Q)}{V(Q)}\Big)\right)^{1-t}\left(\exp\Big(\frac{\widetilde{E}(L,Q)}{V(Q)}\Big)\right)^t,
\end{align}
with equality for any $t\in(0,1)$ if and only if $K=L$.
\end{lem}

\begin{proof}
For $t\in(0,1)$,
by Lemma \eqref{inclusion} and the concavity of the function $\log$, we have
\begin{align}
&\widetilde{E}\big((1-t)\cdot K+_pt\cdot L,Q\big)\\
&\quad=\frac{1}{n}\int_{S^{n-1}}\log\left(\frac{\rho_{(1-t)\cdot K+_pt\cdot L}(u)}{\rho_{Q}(u)}\right)\rho^{n}_{Q}(u)du\nonumber\\
&\quad\geq\frac{1}{n}\int_{S^{n-1}}\log\left(\frac{\rho_{(1-t)\cdot K\tilde{+}_pt\cdot L}(u)}{\rho_{Q}(u)}\right)\rho^{n}_{Q}(u)du\label{BMIq=01-1}\\
&\quad=\frac{1}{np}\int_{S^{n-1}}\log\left(\frac{(1-t)\rho_{K}^p(u)+t\rho_{L}^p(u)}{\rho_{Q}^p(u)}\right)\rho^{n}_{Q}(u)du\nonumber\\
&\quad\geq\frac{1}{np}\int_{S^{n-1}}\left((1-t)\log\left(\frac{\rho_{K}^p(u)}{\rho_{Q}^p(u)}\right)
+t\log\left(\frac{\rho_{L}^p(u)}{\rho_{Q}^p(u)}\right) \right)\rho^{n}_{Q}(u)du\label{BMIq=01-2}\\
&\quad=\frac{1}{n}\int_{S^{n-1}}\left((1-t)\log\left(\frac{\rho_{K}(u)}{\rho_{Q}(u)}\right)+t\log\left(\frac{\rho_{L}(u)}{\rho_{Q}(u)}\right) \right)\rho^{n}_{Q}(u)du\nonumber\\
&\quad=(1-t)\widetilde{E}(K,Q)+t\widetilde{E}(L,Q).\label{BMIq=01-3}
\end{align}
For $t\in(0,1)$, by Lemma \eqref{inclusion} and the fact that the function $\log$ is strictly concave,
we obtain the both equalities in \eqref{BMIq=01-1} and \eqref{BMIq=01-2} hold if and only if $K=L$.
It follows that
\begin{align*}
\exp\left(\frac{\widetilde{E}\big((1-t)\cdot K+_pt\cdot L,Q\big)}{V(Q)}\right)
&\geq \exp\left(\frac{(1-t)\widetilde{E}(K,Q)+t\widetilde{E}(L,Q)}{V(Q)}\right)\\
&=\left(\exp\Big(\frac{\widetilde{E}(K,Q)}{V(Q)}\Big)\right)^{1-t}\left(\exp\Big(\frac{\widetilde{E}(L,Q)}{V(Q)}\Big)\right)^t,
\end{align*}
with equality for any $t\in(0,1)$ if and only if $K=L$.
\end{proof}

\begin{remark}
From \eqref{BMIq=01-3} and the arithmetric-geometric inequality, we have the following result:

Let $Q\in \mathcal{S}_{o}^n$ and $t\in[0,1]$. If $K,L\in\mathcal{K}_o^n$ and $p>0$, then
\begin{align*}
\widetilde{E}\big((1-t)\cdot K+_pt\cdot L,Q\big)
\geq\widetilde{E}(K,Q)^{1-t}\widetilde{E}(L,Q)^t,
\end{align*}
with equality for any $t\in(0,1)$ if and only if $K=L$.
\end{remark}

\begin{lem}
Let $Q\in \mathcal{S}_{o}^n$ and $p>0$. If $K,L\in\mathcal{K}_o^n$, then
\begin{align}\label{BMIq=02}
\left(\exp\Big(\frac{\widetilde{E}\big( K+_p L,Q\big)}{V(Q)}\Big)\right)^p
\geq \left(\exp\Big(\frac{\widetilde{E}(K,Q)}{V(Q)}\Big)\right)^{p}+\left(\exp\Big(\frac{\widetilde{E}(L,Q)}{V(Q)}\Big)\right)^p,
\end{align}
with equality if and only if $K$ and $L$ are dilates.
\end{lem}

\begin{proof}
Let $\bar{K}=\left(\exp\Big(\frac{\widetilde{E}(K,Q)}{V(Q)}\Big)\right)^{-1}K$ and
$\bar{L}=\left(\exp\Big(\frac{\widetilde{E}(L,Q)}{V(Q)}\Big)\right)^{-1}L$,
then $\exp\Big(\frac{\widetilde{E}(\bar{K},Q)}{V(Q)}\Big)=\exp\Big(\frac{\widetilde{E}(\bar{L},Q)}{V(Q)}\Big)=1$.
By \eqref{BMIq=01}, we have
\begin{align}\label{BMIq=02-1}
\exp\left(\frac{\widetilde{E}\big((1-t)\cdot \bar{K}+_pt\cdot \bar{L},Q\big)}{V(Q)}\right)
\geq 1,
\end{align}
with equality for some $t\in(0,1)$ if and only if $\bar{K}=\bar{L}$, that is, $K$ and $L$ are dilates.

Let $t=\left(\exp\Big(\frac{\widetilde{E}(L,Q)}{V(Q)}\Big)\right)^p\Big/
\left(\Big(\exp\Big(\frac{\widetilde{E}(K,Q)}{V(Q)}\Big)\Big)^{p}
+\Big(\exp\Big(\frac{\widetilde{E}(L,Q)}{V(Q)}\Big)\Big)^p\right)$, then
\begin{align*}
(1-t)\cdot \bar{K}+_pt\cdot \bar{L}
=\frac{1}{\left(\Big(\exp\Big(\frac{\widetilde{E}(K,Q)}{V(Q)}\Big)\Big)^{p}
+\Big(\exp\Big(\frac{\widetilde{E}(L,Q)}{V(Q)}\Big)\Big)^p\right)^\frac{1}{p}}
K+_p L.
\end{align*}
Together with \eqref{BMIq=02-1}, we obtain the desired formila \eqref{BMIq=02}.
The equality holds in \eqref{BMIq=02} if and only if the equality holds in \eqref{BMIq=02-1}, that is, $K$ and $L$ are dilates.
\end{proof}

The Minkowski style inequality for $p>0$ is established as follows:

\begin{thm}
Let $Q\in\mathcal{S}_{o}^n$ and $p>0$.
If $K,L\in \mathcal{K}_{o}^n$, then
\begin{align}\label{MIq=0}
\frac{1}{nV(Q)}\int_{S^{n-1}}\log\bigg(\frac{\rho_{L}(u)}{\rho_K(u)}\bigg)\rho^n_{Q}(u)du
\leq\frac{1}{p}\log\frac{\widetilde{V}_{p,0}(K,L,Q)}{V(Q)},
\end{align}
with equality if and only if $K$ and $L$ are dilates.
\end{thm}

\begin{proof}
By \eqref{BMIq=02}, the following function
\begin{align*}
h(t)&=\left(\exp\Big(\frac{\widetilde{E}\big( (1-t)\cdot K+_pt\cdot L,Q\big)}{V(Q)}\Big)\right)^p
-(1-t)\left(\exp\Big(\frac{\widetilde{E}(K,Q)}{V(Q)}\Big)\right)^{p}\\
&\quad-t\left(\exp\Big(\frac{\widetilde{E}(L,Q)}{V(Q)}\Big)\right)^p
\end{align*}
is nonnegative for $t\in[0,1]$ and satisfies $h(0)=h(1)=0$.
Moreover, by \eqref{BMIq=02} again, we obtain
\begin{align*}
h((1-\lambda) t_1+\lambda t_2)\geq (1-\lambda)h(t_1)+\lambda h(t_2),
\end{align*}
for $t_1,t_2\in [0,1]$ and $\lambda\in(0,1)$.
Hence, $h(t)$ is concave for $t\in[0,1]$. Together with \eqref{variationalfq=0},  we have
\begin{align*}
\lim_{t\rightarrow0}\frac{h(t)-h(0)}{t}
&=\left(\exp\Big(\frac{\widetilde{E}\big( K,Q\big)}{V(Q)}\Big)\right)^p \frac{\widetilde{V}_{p,0}(K,L,Q)-V(Q)}{V(Q)}
+\left(\exp\Big(\frac{\widetilde{E}(K,Q)}{V(Q)}\Big)\right)^{p}\\
&\quad-\left(\exp\Big(\frac{\widetilde{E}(L,Q)}{V(Q)}\Big)\right)^p\\
&=\left(\exp\Big(\frac{\widetilde{E}\big( K,Q\big)}{V(Q)}\Big)\right)^p \frac{\widetilde{V}_{p,0}(K,L,Q)}{V(Q)}
-\left(\exp\Big(\frac{\widetilde{E}(L,Q)}{V(Q)}\Big)\right)^p
\geq 0,
\end{align*}
that is, the inequality \eqref{MIq=0} is established.

Furthermore, $\lim_{t\rightarrow0}\frac{h(t)-h(0)}{t}=0$ only if $h(t)$ is identically 0 for $t\in[0,1]$. The latter implies
$(1-t)\cdot K$ and $t\cdot L$ are dilates
by \eqref{BMIq=02}. Hence, the equality in \eqref{MIq=0} holds if and only if $K$ and $L$ are dilates.
\end{proof}

If $Q=K$ in \eqref{MIq=0}, then, by \eqref{Lpduamv-mv}, we obtain the following inequality:
\begin{cor}
If $K,L\in \mathcal{K}_{o}^n$ and $p>0$, then
\begin{align*}
\frac{1}{nV(K)}\int_{S^{n-1}}\log\bigg(\frac{\rho_{L}(u)}{\rho_K(u)}\bigg)\rho^n_{K}(u)du
\leq\frac{1}{p}\log\frac{V_{p}(K,L)}{V(K)},
\end{align*}
%that is,
%\begin{align*}
%\frac{\widetilde{E}(L,K)}{V(K)}
%\leq\frac{1}{p}\log\frac{V_{p}(K,L)}{V(K)},
%\end{align*}
with equality if and only if $K$ and $L$ are dilates.
\end{cor}

The following theorem establishes the uniqueness of the solution to the $L_p$ dual Minkowski problem for $p>0$ and $q=0$.

\begin{thm}\label{uniquenessq=0}
Let $Q\in \mathcal{S}_{o}^n$ and $p>0$. If $K,L\in\mathcal{K}_o^n$
and
\begin{align*}
\widetilde{C}_{p,0}(K,Q,\cdot)=\widetilde{C}_{p,0}(L,Q,\cdot),
\end{align*}
then $K=L$.
\end{thm}

\begin{proof}
By \eqref{Lpdualmv}, $\widetilde{C}_{p,0}(K,Q,\cdot)=\widetilde{C}_{p,0}(L,Q,\cdot)$ and \eqref{Lpdualmvdmv},
we have
\begin{align}\label{uniquenessq=0-1}
\widetilde{V}_{p,0}(K,B,Q)=\widetilde{C}_{p,0}(K,Q,S^{n-1})
=\widetilde{C}_{p,0}(L,Q,S^{n-1})=\widetilde{V}_{p,0}(L,B,Q),
\end{align}
and
\begin{align}
\widetilde{V}_{p,0}(K,L,Q)&=\widetilde{V}_{p,0}(L,L,Q)=V(Q),\label{uniquenessq=0-2}\\
\widetilde{V}_{p,0}(L,K,Q)&=\widetilde{V}_{p,0}(K,K,Q)=V(Q).\label{uniquenessq=0-3}
\end{align}
Combining \eqref{MIq=0} with \eqref{uniquenessq=0-2}, we have
\begin{align}\label{uniquenessq=0-4}
\frac{1}{nV(Q)}\int_{S^{n-1}}\log\bigg(\frac{\rho_{L}(u)}{\rho_K(u)}\bigg)\rho^n_{Q}(u)du
\leq0,
\end{align}
with equality if and only if $K$ and $L$ are dilates. By \eqref{MIq=0} and \eqref{uniquenessq=0-3},
\begin{align}\label{uniquenessq=0-5}
\frac{1}{nV(Q)}\int_{S^{n-1}}\log\bigg(\frac{\rho_{K}(u)}{\rho_L(u)}\bigg)\rho^n_{Q}(u)du
\leq0,
\end{align}
with equality if and only if $K$ and $L$ are dilates.

Combining \eqref{uniquenessq=0-4} with \eqref{uniquenessq=0-5},
we obtain
\begin{align*}
\frac{1}{nV(Q)}\int_{S^{n-1}}\log\bigg(\frac{\rho_{L}(u)}{\rho_K(u)}\bigg)\rho^n_{Q}(u)du
=0.
\end{align*}
It follows that $K$ and $L$ are dilates from the equality condition of \eqref{uniquenessq=0-4}. Together with \eqref{uniquenessq=0-1}, we obtain $K=L$.
\end{proof}

\begin{remark}
Since $L_p$ dual Minkowski problem for $q=0$ and $Q=B$ is $L_p$ Aleksandrov problem (see \cite{HuangLYZ2018}),
then we obtain the uniqueness of the solution to $L_p$ Aleksandrov problem for $p>0$ by Theorem \ref{uniquenessq=0}
for $Q=B$.
\end{remark}

Theorem \ref{uniqueness} is restated as follows:

\begin{thm}\label{uniqueness-0}
Let $q<p$ and $Q\in \mathcal{S}_{o}^n$. If $K,L\in\mathcal{K}_o^n$
and
\begin{align*}
\widetilde{C}_{p,q}(K,Q,\cdot)=\widetilde{C}_{p,q}(L,Q,\cdot),
\end{align*}
then $K=L$.
\end{thm}

\begin{proof}
Theorem \ref{uniqueness-0} is divided into three cases.
Theorem \ref{uniquenesspq0} is the case $q<p$ with $p\neq0$ and $q\neq0$ of  Theorem \ref{uniqueness-0}.
By \eqref{Lpdualcmp=0}, we have Theorem \ref{uniquenessp=0} is the case $p=0$ and $q<0$ of  Theorem \ref{uniqueness-0}.
Furthermore, Theorem \ref{uniquenessq=0} is the case $p>0$ and $q=0$ of  Theorem \ref{uniqueness-0}.
Therefore, Theorem \ref{uniqueness-0} is obtained by Theorem \ref{uniquenesspq0},
Theorem \ref{uniquenessp=0} and Theorem \ref{uniquenessq=0}.
\end{proof}

\begin{remark}
In \cite{LYZ2018}, Lutwak, Yang and Zhang established the uniqueness of the solution to the $L_p$ dual Minkowski problem for the case of polytopes when $q<p$. We extend their result to the case for general convex bodies by constructing some new Minkowski-type inequalities.
\end{remark}

%%%%%%%%%%%%%%%%%%%%%%%%%%%%%%%%%%%%%%%%%%%%%%%%%%%%%%%%%%%%%%%%%%%%%%%%%%%%%
\vskip 0.5cm
\section{Continuity}\label{}
\vskip 0.3cm
%%%%%%%%%%%%%%%%%%%%%%%%%%%%%%%%%%%%%%%%%%%%%%%%%%%%%%%%%%%%%%%%%%%%%%%%%%%%%%

In this section, we will consider the continuity of the solution to the $L_p$ dual Minkowski problem.
For the case $q<0\leq p$, some lemmas are needed.
Motivated by Zhao \cite{Zhao1}, we obtain the following lemma:

\begin{lem}\label{polarbound}
Let $q<0$, $Q\in\mathcal{S}_{o}^n$ and $K_i\in\mathcal{K}_{o}^n$ for $i=1,2,\cdots$. If
the sequence $\{\widetilde{V}_q(K_i,Q)\}$ is bounded from above, then there exists $M=M(Q)> 1$ such that
\begin{align*}
K^*_i\subseteq MB,
\end{align*}
for all $i=1,2,\cdots$.
\end{lem}

\begin{proof}
Since $Q\in\mathcal{S}_{o}^n$, there exists a constant $r_Q>0$ such that $r_QB\subseteq Q$. Then, for each $u\in S^{n-1}$,
\begin{align}\label{bound-1}
\rho_Q(u)\geq r_Q.
\end{align}
Since $\rho_{K^*_i}$ is continuous function on compact set $S^{n-1}$, then there exists $v_i\in S^{n-1}$ such that
\begin{align*}
\rho_{K^*_i}(v_i)=\max\{\rho_{K^*_i}(v): v\in S^{n-1}\}.
\end{align*}
Let $\xi_+=\max\{\xi, 0\}$. From $K_i\in\mathcal{K}_{o}^n$ and the definition of the support function, we have
\begin{align}\label{bound-2}
h_{K^*_i}(u)\geq (u\cdot v_i)_+\rho_{K^*_i}(v_i),
\end{align}
for all $u\in S^{n-1}$.

By the rotational invariance of the spherical Lebesgue measure, the integral
\begin{align}\label{bound-2-1}
\frac{1}{n}\int_{S^{n-1}}(u\cdot v)_+^{-q}du
\end{align}
is independent of the choice of $v\in S^{n-1}$.
Since the spherical
Lebesgue measure is not concentrated in any closed hemisphere, then the integral \eqref{bound-2-1} is positive for all $v\in S^{n-1}$.
Thus, there exists a constant $m_0>0$ such that
\begin{align}\label{bound-3}
m_0=\frac{1}{n}\int_{S^{n-1}}(u\cdot v)_+^{-q}du,\quad v\in S^{n-1}.
\end{align}

Combining \eqref{bound-1}, \eqref{bound-2}, \eqref{bound-3} with $q<0$, we have
\begin{align*}
\widetilde{V}_q(K_i,Q)
=\frac{1}{n}\int_{S^{n-1}}\rho^q_{K_i}(u)\rho^{n-q}_{Q}(u)du
&\geq \frac{r_Q^{n-q}}{n}\int_{S^{n-1}}h^{-q}_{K^*_i}(u)du\\
&\geq \frac{r_Q^{n-q}}{n}\int_{S^{n-1}}(u\cdot v_i)_+^{-q}\rho_{K^*_i}(v_i)^{-q}du\\
&=m_0r_Q^{n-q}\rho_{K^*_i}(v_i)^{-q}.
\end{align*}
It follows that
\begin{align*}
\rho_{K^*_i}(v_i)\leq m_0^\frac{1}{q}r_Q^{\frac{n-q}{q}}\widetilde{V}_q(K_i,Q)^\frac{1}{-q},
\end{align*}
for all $i$.

Since $\{\widetilde{V}_q(K_i,Q)\}$ is bounded from above and $q<0$, then
then there exists $M=M(Q)> 1$ such that
\begin{align*}
K^*_i\subseteq MB,
\end{align*}
for all $i=1,2,\cdots$.
\end{proof}

Let $\omega_\delta(u)=\{v\in S^{n-1}: v\cdot u>\delta\}$ for $u\in S^{n-1}$ and $0<\delta<1$, then
$\omega_\delta(u)$ is an open set of $S^{n-1}$.
In \cite{Zhao1}, Zhao applied the following result to prove the existence of the dual Minkowski problem for negative indices. For the sake of completeness, we provide the proof of the following lemma.

\begin{lem}[\cite{Zhao1}]\label{oboundary}
Suppose the sequence $\{K_i\}\subseteq \mathcal{K}_{o}^n$ converges to a compact convex set $K$ with $o\in\partial K$
and $\mu$ is a non-zero finite Borel measure on $S^{n-1}$ which is not concentrated
in any closed hemisphere of $S^{n-1}$, then there exist a constant $0<\delta_0<1$ and $u_0\in S^{n-1}$ such that
\begin{align*}
\mu(\omega_{\delta_0}(u_0))>0,
\end{align*}
and the sequence $\{\rho_{K_{i}}\}$ uniformly converges to $0$ on $\omega_{\delta_0}(u_0)$.
\end{lem}

\begin{proof}
Since $o\in\partial K$,
then there exists $u_0\in S^{n-1}$ such that $h_{K}(u_0)=0$.
Since the sequence $\{K_i\}$ converges to $K$, then
\begin{align}\label{ssc}
\lim_{i\rightarrow +\infty}h_{K_{i}}(u_0)=h_{K}(u_0)=0.
\end{align}
For each $v\in \omega_\delta(u_0)$, we have $\rho_{K_{i}}(v)v\in\partial K_i$. Therefore, by \eqref{supportf},
\begin{align*}
h_{K_{i}}(u_0)\geq (v\cdot u_0)\rho_{K_{i}}(v)>\delta \rho_{K_{i}}(v),\quad v\in \omega_\delta(u_0).
\end{align*}
It follows that
\begin{align}\label{rsc}
\rho_{K_{i}}(v)<\frac{1}{\delta}h_{K_{i}}(u_0),
\end{align}
for all $v\in \omega_\delta(u_0)$. By \eqref{ssc} and \eqref{rsc},
$\{\rho_{K_{i}}\}$ uniformly converges to $0$ on $\omega_\delta(u_0)$.

By the monotone convergence theorem and the assumption that $\mu$ is not concentrated in any closed hemisphere, we obtain
\begin{align*}
\lim_{\delta\rightarrow 0}
\mu(\omega_{\delta}(u_0))=\mu(\{v\in S^{n-1}: v\cdot u_0>0\})>0.
\end{align*}
There exists a constant $0<\delta_0<1$ such that
\begin{align*}
\mu(\omega_{\delta_0}(u_0))>0,
\end{align*}
and the sequence $\{\rho_{K_{i}}\}$ uniformly converges to $0$ on $\omega_{\delta_0}(u_0)$.
\end{proof}

The weak convergence of the sequence of $L_p$ dual curvature measures implies that the sequence of the corresponding total measures is bounded from above and below.

\begin{lem}\label{pmixedvolumeb}
Suppose $q<p$, $Q\in\mathcal{S}_{o}^n$ and $K_i\in\mathcal{K}_{o}^n$ for each $i=0,1,\cdots$.
If the sequence $\{\widetilde{C}_{p,q}(K_i,Q,\cdot)\}$ converges to $\widetilde{C}_{p,q}(K_0,Q,\cdot)$ weakly, then
there exist two constants $0<m_1\leq m_2$ such that, for all $i=1,2,\cdots$,
\begin{align*}
m_1\leq \widetilde{V}_{p,q}(K_i,B,Q)\leq m_2.
\end{align*}
\end{lem}

\begin{proof}
Since the assume that $\{\widetilde{C}_{p,q}(K_i,Q,\cdot)\}$ converges to $\widetilde{C}_{p,q}(K_0,Q,\cdot)$ weakly,
then the sequence of the corresponding total measures is also convergent, that is,
\begin{align*}
\widetilde{V}_{p,q}(K_i,B,Q)\rightarrow\widetilde{V}_{p,q}(K_0,B,Q), \quad \text{as}~i\rightarrow+\infty.
\end{align*}
By $K_0\in\mathcal{K}_{o}^n$, $Q\in\mathcal{S}_{o}^n$ and \eqref{Lpdualmv}, we have $\widetilde{V}_{p,q}(K_0,B,Q)>0$. Thus, there exist
$0<\varepsilon_0<\widetilde{V}_{p,q}(K_0,B,Q)$ and positive integer $N_0$ such that
\begin{align*}
\widetilde{V}_{p,q}(K_0,B,Q)-\varepsilon_0\leq\widetilde{V}_{p,q}(K_i,B,Q)\leq\widetilde{V}_{p,q}(K_0,B,Q)+\varepsilon_0
\end{align*}
for all $i>N_0$.

Choosing $m_1=\min\{\widetilde{V}_{p,q}(K_0,B,Q)-\varepsilon_0, \widetilde{V}_{p,q}(K_1,B,Q), \cdots,\widetilde{V}_{p,q}(K_{N_0},B,Q)\}$ and $m_2=\max\{\widetilde{V}_{p,q}(K_0,B,Q)+\varepsilon_0, \widetilde{V}_{p,q}(K_1,B,Q), \cdots,\widetilde{V}_{p,q}(K_{N_0},B,Q)\}$, we obtain,
\begin{align*}
m_1\leq \widetilde{V}_{p,q}(K_i,B,Q)\leq m_2,
\end{align*}
for all $i$.
\end{proof}

Next, we will estimate the bounds of the sequence of dual mixed volumes.

\begin{lem}\label{dualmvabb}
Suppose $q<0\leq p$, $Q\in\mathcal{S}_{o}^n$ and $K_i\in\mathcal{K}_{o}^n$ for each $i=0,1,\cdots$.
If $\{\widetilde{C}_{p,q}(K_i,Q,\cdot)\}$ converges to $\widetilde{C}_{p,q}(K_0,Q,\cdot)$ weakly, then
there exist two constant $0<m_3\leq m_4$ such that
\begin{align}\label{dualmixedvb}
m_3\leq\widetilde{V}_q(K_i,Q)\leq m_4,
\end{align}
for all $i=1,2,\cdots$.
\end{lem}

\begin{proof}
Let $\bar{K}_i=\widetilde{V}_q(K_i,Q)^{-\frac{1}{q}}K_i$, then $\widetilde{V}_q(\bar{K}_i,Q)=1$. By Lemma \ref{polarbound},
there exists $M=M(Q)>1$ such that
\begin{align*}
\bar{K}_i^*\subseteq MB,
\end{align*}
for all $i=1,2,\cdots$.
Then,
\begin{align*}
K_i^*=\widetilde{V}_q(K_i,Q)^{-\frac{1}{q}}\bar{K}_i^*\subseteq \widetilde{V}_q(K_i,Q)^{-\frac{1}{q}}M B.
\end{align*}
It follows from \eqref{Lpdualcm2} that
\begin{align*}
\widetilde{V}_{p,q}(K_i,B,Q)
&=\int_{S^{n-1}}h_{K_i}^{-p}(u)d\widetilde{C}_{q}(K_i,Q,u)\\
&=\int_{S^{n-1}}\rho_{K_i^*}^{p}(u)d\widetilde{C}_{q}(K_i,Q,u)
\leq M^p \widetilde{V}_q(K_i,Q)^{\frac{q-p}{q}}.
\end{align*}
Together with Lemma \ref{pmixedvolumeb} and $q<0\leq p$, we obtain
\begin{align*}
\widetilde{V}_q(K_i,Q)\geq (m_1 M^{-p})^{\frac{q}{q-p}},
\end{align*}
for all $i=1,2,\cdots$. Hence, we could take $m_3=(m_1 M^{-p})^{\frac{q}{q-p}}$ for $q<0\leq p$.

When $q<0< p$, by \eqref{MIp},
\begin{align*}
\widetilde{V}_{p,q}(K_i,B,Q)^q\leq \widetilde{V}_q(K_i,Q)^{q-p}\widetilde{V}_q(B,Q)^p.
\end{align*}
Hence, together with $q<0< p$ and Lemma \ref{pmixedvolumeb}, we have
\begin{align*}
\widetilde{V}_q(K_i,Q)\leq \widetilde{V}_{p,q}(K_i,B,Q)^\frac{q}{q-p}\widetilde{V}_q(B,Q)^\frac{-p}{q-p}
\leq m_2^\frac{q}{q-p}\widetilde{V}_q(B,Q)^\frac{-p}{q-p},
\end{align*}
for all $i=1,2,\cdots$. Thus, we could take $m_4=m_2^\frac{q}{q-p}\widetilde{V}_q(B,Q)^\frac{-p}{q-p}$ when $q<0< p$.

From \eqref{Lpdualmv} and \eqref{Lpdualcmp=0},
\begin{align*}
\widetilde{V}_{0,q}(K_i,B,Q)=\widetilde{C}_{0,q}(K_i,Q,S^{n-1})=\widetilde{C}_{q}(K_i,Q,S^{n-1})= \widetilde{V}_q(K_i,Q).
\end{align*}
Together with Lemma \ref{pmixedvolumeb}, we could take $m_4=m_2$ when $p=0$ and $q<0$.

By choosing $m_3=(m_1 M^{-p})^{\frac{q}{q-p}}$ and $m_4=m_2^\frac{q}{q-p}\widetilde{V}_q(B,Q)^\frac{-p}{q-p}$ for $q<0\leq p$, we complete the proof of Lemma \ref{dualmvabb}.
\end{proof}

We will show that
the weak convergence of the sequence of the $L_p$ dual curvature measures for $q<0\leq p$ implies that the sequence of the corresponding convex bodies
is bounded from above and below.

\begin{lem}\label{convebsaboveb1}
Suppose $q<0\leq p$, $Q\in\mathcal{S}_{o}^n$ and $K_i\in\mathcal{K}_{o}^n$ for each $i=0,1,\cdots$.
If $\{\widetilde{C}_{p,q}(K_i,Q,\cdot)\}$ converges to $\widetilde{C}_{p,q}(K_0,Q,\cdot)$ weakly, then
the sequence $\{K_i\}$ is bounded from above.
\end{lem}

\begin{proof}
Let $\bar{K}_i=\widetilde{V}_q(K_i,Q)^{-\frac{1}{q}}K_i$, then $\widetilde{V}_q(\bar{K}_i,Q)=1$.
By Lemma \ref{polarbound},
there exists $M=M(Q)>1$ such that
\begin{align*}
\bar{K}_i^*\subseteq MB
\end{align*}
for all $i=1,2,\cdots$, that is, the sequence $\{\bar{K}_i^*\}$ is bounded from above.
Together with the Blaschke's selection theorem, without loss of generality, we could assume that $\{\bar{K}^*_{i}\}$ is a convergent sequence whose limit is a compact convex set $K$ with $o\in K$.

Suppose that the sequence $\{\bar{K}_i\}$ is not bounded from above.
Let
\begin{align*}
\bar{R}_i=h_{\bar{K}_i}(v_i)=\max\{h_{\bar{K}_i}(v): v\in S^{n-1}\}, \quad v_i\in S^{n-1},
\end{align*}
then we may assume that $\bar{R}_i\leq \bar{R}_j$ for $i\leq j$ and $\lim_{i\rightarrow+\infty}\bar{R}_i=+\infty$.
Since $S^{n-1}$ is a compact set, without loss of generality, we could assume that the sequence $\{v_i\}$ tends to $v_0\in S^{n-1}$.
By \eqref{supportradialf}, we have
\begin{align*}
\bar{R}^{-1}_iv_i=\rho_{\bar{K}^*_i}(v_i)v_i\in \partial\bar{K}^*_i,
\end{align*}
for $v_i\in S^{n-1}$. Since $\lim_{i\rightarrow+\infty}\bar{R}_i=+\infty$ and the sequence $\{\bar{K}^*_{i}\}$ converges to compact convex set $K$, then we obtain $o\in\partial K$.

Since $\widetilde{C}_{p,q}(K_0,Q,\cdot)$
is not concentrated in any closed hemisphere, by Lemma \ref{oboundary}, there exist a constant $0<\delta_0<1$ and $u_0\in S^{n-1}$ such that
\begin{align*}
\widetilde{C}_{p,q}(K_0,Q,\omega_{\delta_0}(u_0))>0,
\end{align*}
and the sequence $\{\rho_{\bar{K}_i^*}\}$ uniformly converges to $0$ on $\omega_{\delta_0}(u_0)$.

Since that $\{\widetilde{C}_{p,q}(K_i,Q,\cdot)\}$ converges to $\widetilde{C}_{p,q}(K_0,Q,\cdot)$ weakly and that $\omega_{\delta_0}(u_0)$ is an open set of $S^{n-1}$, we have
\begin{align*}
\liminf_{i\rightarrow+\infty} \widetilde{C}_{p,q}(K_i,Q,\omega_{\delta_0}(u_0))\geq \widetilde{C}_{p,q}(K_0,Q,,\omega_{\delta_0}(u_0))>0.
\end{align*}
Hence,we can choose a subsequence of the sequence $\{K_{i}\}$, denoted
again by $\{K_{i}\}$, such that
\begin{align}\label{dualmeasure}
\lim_{i\rightarrow+\infty}  \widetilde{C}_{p,q}(K_i,Q,\omega_{\delta_0}(u_0))\geq \widetilde{C}_{p,q}(K_0,Q,,\omega_{\delta_0}(u_0))>0.
\end{align}

For $q<0<p$,
combining \eqref{dualmeasure} with the fact that $\{\rho_{\bar{K}_i^*}\}$ uniformly converges to $0$ on $\omega_{\delta_0}(u_0)$, we have
\begin{align*}
\int_{\omega_{\delta_0}(u_0)}\rho_{\bar{K}_i^*}^{-p}(u)d\widetilde{C}_{p,q}(K_i,Q,u)\rightarrow+\infty,
\quad\text{as}~i\rightarrow+\infty.
\end{align*}
Together with \eqref{Lpdcmh} and \eqref{dualmixedvb},  we obtain
\begin{align*}
\widetilde{V}_q(\bar{K}_i,Q)=\widetilde{V}_{p,q}(\bar{K}_i,\bar{K}_i,Q)
&=\int_{S^{n-1}}h_{\bar{K}_i}^p(u)d\widetilde{C}_{p,q}(\bar{K}_i,Q,u)\\
&=\widetilde{V}_q(K_i,Q)^{\frac{p-q}{q}}\int_{S^{n-1}}\rho_{\bar{K}_i^*}^{-p}(u)d\widetilde{C}_{p,q}(K_i,Q,u)\\
&\geq m_4^{\frac{p-q}{q}}\int_{\omega_{\delta_0}(u_0)}\rho_{\bar{K}_i^*}^{-p}(u)d\widetilde{C}_{p,q}(K_i,Q,u)\\
&\rightarrow+\infty,
\end{align*}
as $i\rightarrow+\infty$.
This contradicts $\widetilde{V}_q(\bar{K}_i,Q)=1$ for all $i$.
Hence, the sequence $\{\bar{K}_i\}$ is bounded from above for $q<0<p$.

For $p=0$ and $q<0$, from \eqref{MIp=0}, we have
\begin{align}\label{convebsaboveb1-1}
\int_{S^{n-1}}\log h_{\bar{K}_i}(u)d\widetilde{C}_q(\bar{K}_i,Q,u)
\leq\frac{1}{q}\log\frac{1}{\widetilde{V}_q(B,Q)},
\end{align}
for all $i$.
By \eqref{Lpdualcmp=0}, \eqref{dualmeasure} and the fact that $\{\rho_{\bar{K}_i^*}\}$ uniformly converges to $0$ on $\omega_{\delta_0}(u_0)$,
\begin{align}\label{convebsaboveb1-2}
\int_{\omega_{\delta_0}(u_0)}\log\rho_{\bar{K}_i^*}(u)d\widetilde{C}_{q}(K_i,Q,u)\rightarrow-\infty,
\quad\text{as}~i\rightarrow+\infty.
\end{align}
From \eqref{dualmixedvb} and $\bar{K}_i^*\subseteq MB$ for all $i$,
\begin{align*}
\int_{S^{n-1}}\log h_{\bar{K}_i}(u)d\widetilde{C}_q(\bar{K}_i,Q,u)
&=-\frac{1}{\widetilde{V}_q(K_i,Q)}\int_{S^{n-1}}\log\rho_{\bar{K}_i^*}(u)d\widetilde{C}_q(K_i,Q,u)\\
&\geq-\frac{1}{\widetilde{V}_q(K_i,Q)}
\int_{\omega_{\delta_0}(u_0)}\log\rho_{\bar{K}_i^*}(u)d\widetilde{C}_q(K_i,Q,u)-\log M\\
&\geq-\frac{1}{m_4}
\int_{\omega_{\delta_0}(u_0)}\log\rho_{\bar{K}_i^*}(u)d\widetilde{C}_q(K_i,Q,u)-\log M,
\end{align*}
for sufficiently large $i$. Together with \eqref{convebsaboveb1-2}, we obtain
\begin{align*}
\int_{S^{n-1}}\log h_{\bar{K}_i}(u)d\widetilde{C}_q(\bar{K}_i,Q,u)\rightarrow+\infty,
\quad\text{as}~i\rightarrow+\infty.
\end{align*}
This is a contradiction to \eqref{convebsaboveb1-1}.
Hence, the sequence $\{\bar{K}_i\}$ is bounded from above for $p=0$ and $q<0$.

By \eqref{dualmixedvb}, we have
\begin{align*}
K_i=\widetilde{V}_q(K_i,Q)^{\frac{1}{q}}\bar{K}_i\subseteq m_3^{\frac{1}{q}}\bar{K}_i,
\end{align*}
for all $i$.
Therefore, together with the fact the sequence $\{\bar{K}_i\}$ is bounded from above for $q<0\leq p$,
we obtain the sequence $\{K_i\}$ is also bounded from above for $q<0\leq p$.
\end{proof}

\begin{lem}\label{convexbsbelowb1}
Suppose $q<0\leq p$, $Q\in\mathcal{S}_{o}^n$ and $K_i\in\mathcal{K}_{o}^n$ for each $i=0,1,\cdots$.
If $\{\widetilde{C}_{p,q}(K_i,Q,\cdot)\}$ converges to $\widetilde{C}_{p,q}(K_0,Q,\cdot)$ weakly, then
there exists a constant $m_5>0$ such that
\begin{align*}
m_5B\subseteq K_i,
\end{align*}
for all $i=1,2,\cdots$.
\end{lem}

\begin{proof}
Let $\bar{K}_i=\widetilde{V}_q(K_i,Q)^{-\frac{1}{q}}K_i$, then $\widetilde{V}_q(\bar{K}_i,Q)=1$. By Lemma \ref{polarbound},
there exists $M=M(Q)>1$ such that
\begin{align*}
\bar{K}_i^*\subseteq MB,
\end{align*}
for all $i=1,2,\cdots$.
Together with \eqref{dualmixedvb}, we have
\begin{align*}
K_i^*=\widetilde{V}_q(K_i,Q)^{-\frac{1}{q}}\bar{K}_i^*\subseteq \widetilde{V}_q(K_i,Q)^{-\frac{1}{q}}MB
\subseteq (m_4^{-\frac{1}{q}}M)B,
\end{align*}
that is,
\begin{align*}
(m_4^{\frac{1}{q}}M^{-1})B\subseteq K_i,
\end{align*}
for all $i=1,2,\cdots$.
Choosing $m_5=m_4^{\frac{1}{q}}M^{-1}$, we have $m_5B\subseteq K_i$ for all $i=1,2,\cdots$.
\end{proof}

\begin{lem}[\cite{LYZ2018}]\label{dualcmwc1}
Suppose $p,q\in\mathbb{R}$ and $Q\in\mathcal{S}_{o}^n$.  If $K_i\in\mathcal{K}_{o}^n$ with $K_i\rightarrow K_0\in\mathcal{K}_{o}^n$,
then $\widetilde{C}_{p,q}(K_i,Q,\cdot)\rightarrow\widetilde{C}_{p,q}(K_0,Q,\cdot)$, weakly.
\end{lem}

Theorem \ref{thm1-1} is restate as follows:

\begin{thm}\label{thm1-01}
Let $q<0\leq p$, $Q\in\mathcal{S}_{o}^n$ and $K_i\in\mathcal{K}_{o}^n$ for each $i=0,1,\cdots$.
If $\{\widetilde{C}_{p,q}(K_i,Q,\cdot)\}$ converges to $\widetilde{C}_{p,q}(K_0,Q,\cdot)$ weakly,
then $\{K_i\}$ converges to $K_0$ in the Hausdorff metric.
\end{thm}

\begin{proof}
Assume that the sequence $\{K_i\}$ does not converge to $K_0$.
Then, there exist $\varepsilon_1 >0$ and a subsequence of $\{K_i\}$, denoted again by $\{K_i\}$, such that
\begin{align*}
||h_{K_i}-h_{K_0}||\geq \varepsilon_1
\end{align*}
for all $i$.

By Lemma \ref{convebsaboveb1} and the Blaschke's selection theorem,
we deduce that $\{K_{i}\}$ has a convergent subsequence, denoted again by $\{K_i\}$, with $\lim_{i\rightarrow+\infty}K_{i}=K'_0$.
Then, $K'_0\neq K_0$. From Lemma \ref{convexbsbelowb1}, it follows that $K'_0\in\mathcal{K}_{o}^n$. By Lemma \ref{dualcmwc1},
$\{\widetilde{C}_{p,q}(K_i,Q,\cdot)\}$ converges to $\widetilde{C}_{p,q}(K'_0,Q,\cdot)$ weakly.

Since $\{\widetilde{C}_{p,q}(K_i,Q,\cdot)\}$ converges to $\widetilde{C}_{p,q}(K_0,Q,\cdot)$ weakly, then
\begin{align*}
\widetilde{C}_{p,q}(K'_0,Q,\cdot)=\widetilde{C}_{p,q}(K_0,Q,\cdot).
\end{align*}
By Theorem \ref{uniqueness}, we obtain $K'_0=K_0$. This is a contradiction to $K'_0\neq K_0$.
Therefore,
the sequence $\{K_i\}$ converges to $K_0$ in the Hausdorff metric.
\end{proof}

\begin{remark}
The case $p=0$ and $Q=B$ of Theorem \ref{thm1-01} is due to Wang, Fang and Zhou \cite{WFZ}.
\end{remark}

By the arguments as in the proof of Theorem \ref{thm1-01}, we obtain the following theorems.

\begin{thm}\label{thm1-2}
Let $q_0<0\leq p$, $Q\in\mathcal{S}_{o}^n$ and $K_i\in\mathcal{K}_{o}^n$ for each $i=0,1,\cdots$.
If $\widetilde{C}_{p,q_i}(K_i,Q,\cdot)=\widetilde{C}_{p,q_0}(K_0,Q,\cdot)$ with $\lim_{i\rightarrow+\infty}q_i=q_0$,
then $\{K_i\}$ converges to $K_0$ in the Hausdorff metric.
\end{thm}

\begin{thm}\label{thm1-3}
Let $q<0< p_0$, $Q\in\mathcal{S}_{o}^n$ and $K_i\in\mathcal{K}_{o}^n$ for each $i=0,1,\cdots$.
If $\widetilde{C}_{p_i,q}(K_i,Q,\cdot)=\widetilde{C}_{p_0,q}(K_0,Q,\cdot)$ with $\lim_{i\rightarrow+\infty}p_i=p_0$,
then $\{K_i\}$ converges to $K_0$ in the Hausdorff metric.
\end{thm}

Now, we consider the continuity of the solution to the $L_p$ dual Minkowski problem for $p\geq 1$ and $0\leq q<p$.

\begin{lem}\label{uniformlyc}
Suppose $p\geq 1$, $q\in\mathbb{R}$, $Q\in\mathcal{S}_{o}^n$ and $K_i\in\mathcal{K}_{o}^n$ for each $i=0,1,\cdots$.
Let
\begin{align*}
f_i(u)=\int_{S^{n-1}}(u\cdot v)_+^pd\widetilde{C}_{p,q}(K_i,Q,v)
\end{align*}
for $u\in S^{n-1}$ and $i=0,1,\cdots$.
If $\{\widetilde{C}_{p,q}(K_i,Q,\cdot)\}$ converges to $\widetilde{C}_{p,q}(K_0,Q,\cdot)$ weakly,
then the sequence $\{f_i\}$ converges to $f_0$ uniformly on $S^{n-1}$.

\end{lem}

\begin{proof}
It is clear that $f_i^\frac{1}{p}$ is sublinear function for each $i$.
Then, each $f_i^\frac{1}{p}$ is support function of some convex body.
Since $\{\widetilde{C}_{p,q}(K_i,Q,\cdot)\}$ converges to $\widetilde{C}_{p,q}(K_0,Q,\cdot)$ weakly, then
the sequence $\{f_i^\frac{1}{p}\}$ is converges pointwise to $f_0^\frac{1}{p}$ on $S^{n-1}$.
Together with that pointwise and uniform convergence are equivalent for support functions on $S^{n-1}$ in \cite{Schneider},
we obtain the sequence $\{f_i\}$ converges to $f_0$ uniformly on $S^{n-1}$.
\end{proof}

The weak convergence of the sequence of the $L_p$ dual curvature measures for $p\geq 1$ and $0\leq q<p$ implies that the sequence of dual mixed volumes of the corresponding convex bodies is bounded from above and below.

\begin{lem}\label{cdsb-dmvsabb}
Let $p\geq 1$ and $0\leq q<p$, $Q\in\mathcal{S}_{o}^n$ and  $K_i\in\mathcal{K}_{o}^n$ for each $i=0,1,\cdots$.
If $\{\widetilde{C}_{p,q}(K_i,Q,\cdot)\}$ converges to $\widetilde{C}_{p,q}(K_0,Q,\cdot)$ weakly,
then the sequence $\{K_i\}$ is bounded from above, and there exist two constants $0<m_{7}\leq m_{8}$ such that
\begin{align*}
m_{6}\leq \widetilde{V}_q(K_i,Q) \leq m_{7}
\end{align*}
for all $i=1,2,\cdots$.
\end{lem}

\begin{proof}
For $0<q<p$, by \eqref{MIp},
\begin{align*}
\widetilde{V}_{p,q}(K_i,B,Q)^q\geq \widetilde{V}_q(K_i,Q)^{q-p}\widetilde{V}_q(B,Q)^p,
\end{align*}
for all $i$.
By Lemma \ref{pmixedvolumeb}, we have
\begin{align*}
\widetilde{V}_q(K_i,Q)\geq \widetilde{V}_{p,q}(K_i,B,Q)^\frac{q}{q-p}\widetilde{V}_q(B,Q)^\frac{p}{p-q}
\geq m_2^\frac{q}{q-p}\widetilde{V}_q(B,Q)^\frac{p}{p-q},
\end{align*}
for all $i$. Let $m_6=m_2^\frac{q}{q-p}\widetilde{V}_q(B,Q)^\frac{p}{p-q}$ for $0\leq q<p$.

Since $\widetilde{C}_{p,q}(K_0,Q,\cdot)$ is not concentrated on a great subsphere of $S^{n-1}$, then
$f_0(u)=\int_{S^{n-1}}(u\cdot v)_+^pd\widetilde{C}_{p,q}(K_0,Q,v)$ has a positive lower bound on $S^{n-1}$.
By Lemma \ref{uniformlyc}, the sequence $\{f_i\}$ has a positive lower bound on $S^{n-1}$,
that is, there exists a constant $m_{8}>0$ such that
\begin{align*}
\int_{S^{n-1}}(u\cdot v)_+^pd\widetilde{C}_{p,q}(K_i,Q,v)\geq\frac{1}{m_{8}^p}
\end{align*}
for all $u\in S^{n-1}$ and $i=1,2,\cdots$.
Let $K_i'=\widetilde{V}_q(K_i,Q)^{-\frac{1}{p}}K_i$, then $K_i=\widetilde{V}_q(K'_i,Q)^{\frac{1}{p-q}}K'_i$.
Thus, by \eqref{Lpdcmh} and \eqref{Lpdualcm2}, we have
\begin{align*}
\frac{{R'}_i^p}{m_{8}^p}\leq {R'}_i^p\int_{S^{n-1}}(u'_i\cdot v)_+^pd\widetilde{C}_{p,q}(K_i,Q,v)
&=\frac{{R'}_i^p}{\widetilde{V}_q(K'_i,Q)}\int_{S^{n-1}}(u'_i\cdot v)_+^pd\widetilde{C}_{p,q}(K'_i,Q,v)\\
&\leq\frac{1}{\widetilde{V}_q(K'_i,Q)}\int_{S^{n-1}}h_{K'_i}(v)^pd\widetilde{C}_{p,q}(K'_i,Q,v)\\
&=\frac{1}{\widetilde{V}_q(K'_i,Q)}\int_{S^{n-1}}d\widetilde{C}_{q}(K'_i,Q,v)\\
&=1,
\end{align*}
where $R'_i=\rho_{K'_i}(u'_i)=\max\{\rho_{K'_i}(u): u\in S^{n-1}\}$ for $u'_i\in S^{n-1}$ and $i=1,2,\cdots$.
It follows that ${R'}_i\leq m_{8}$, that is,
\begin{align}\label{cdsb-dmvsabb-1}
K_i'\subseteq m_{8}B,
\end{align}
for all $i=1,2,\cdots$.
Since $K_i'=\widetilde{V}_q(K_i,Q)^{-\frac{1}{p}}K_i$, we obtain
\begin{align*}
\widetilde{V}_q(K_i,Q)^\frac{p-q}{p}
=\widetilde{V}_q(K'_i,Q)\leq m_{8}^q\widetilde{V}_q(B,Q),
\end{align*}
that is,
\begin{align*}
\widetilde{V}_q(K_i,Q)
\leq \big(m_{8}^q\widetilde{V}_q(B,Q)\big)^\frac{p}{p-q},
\end{align*}
for all $i$.
Let $m_7=\big(m_{8}^q\widetilde{V}_q(B,Q)\big)^\frac{p}{p-q}$  for $0\leq q<p$.

By $K_i'=\widetilde{V}_q(K_i,Q)^{-\frac{1}{p}}K_i$ and \eqref{cdsb-dmvsabb-1}, we have
\begin{align*}
K_i=\widetilde{V}_q(K_i,Q)^{\frac{1}{p}}K'_i\subseteq  m_{8}\widetilde{V}_q(K_i,Q)^{\frac{1}{p}}B
\subseteq  m_{8}m_7^\frac{1}{p}B,
\end{align*}
for all $i=1,2,\cdots$.
Therefore, the sequence $\{K_i\}$ is bounded from above.
\end{proof}

\begin{lem}\label{Kinteriorp}
If $p\geq 1$ and $0\leq q<p$, $Q\in\mathcal{S}_{o}^n$ and the sequence $\{K_i\}\subseteq \mathcal{K}_{o}^n$ converges to compact convex body $K$ in $\mathbb{R}^n$ with that $\{\widetilde{C}_{p,q}(K_i,Q,S^{n-1})\}$ is bounded from above, then
$K\in \mathcal{K}_{o}^n$.
\end{lem}

\begin{proof}
By $Q\in\mathcal{S}_{o}^n$, there exist two constants $0<r_{Q}<R_{Q}$ such that
\begin{align*}
r_{Q}<\rho_{Q}(u)<R_{Q},
\end{align*}
for all $u\in S^{n-1}$.
Since the sequence $\{K_i\}\subseteq \mathcal{K}_{o}^n$ converges to compact convex body $K$, then $o\in K$.

Suppose that $o\in\partial K$. Since the spherical Lebesgue measure is not concentrated
in any closed hemisphere of $S^{n-1}$, then,
by Lemma \ref{oboundary}, we obtain that
there exist a constant $0<\delta_0<1$ and $u_0\in S^{n-1}$ such that the spherical Lebesgue measure of $\omega_{\delta_0}(u_0)\subseteq S^{n-1}$ is greater than 0
and the sequence $\{\rho_{K_{i}}\}$ converges to $0$ uniformly on $\omega_{\delta_0}(u_0)$.
By \eqref{Lpdcmintegral}, we have
\begin{align*}
\widetilde{C}_{p,q}(K_i,Q,S^{n-1})
&=\frac{1}{n}\int_{S^{n-1}}h_{K_i}(\alpha_{K_i}(u))^{-p}\rho^q_{K_i}(u)\rho^{n-q}_{Q}(u)du\\
&=\frac{1}{n}\int_{S^{n-1}}(u\cdot\alpha_{K_i}(u))^{-p}\rho^{q-p}_{K_i}(u)\rho^{n-q}_{Q}(u)du\\
&\geq\frac{\min\{r^{n-q}_{Q}, R^{n-q}_{Q}\}}{n}\int_{S^{n-1}}\rho^{q-p}_{K_i}(u)du\\
&\geq\frac{\min\{r^{n-q}_{Q}, R^{n-q}_{Q}\}}{n}\int_{\omega_{\delta_0}(u_0)}\rho^{q-p}_{K_i}(u)du\\
&\rightarrow+\infty,
\end{align*}
as $i\rightarrow+\infty$. This contradicts $\{\widetilde{C}_{p,q}(K_i,Q,S^{n-1})\}$ is bounded from above. Therefore, we obtain $o\in\text{int}K$, that is, $K\in \mathcal{K}_{o}^n$.
\end{proof}

Theorem \ref{thm2-1} is restate as follows:

\begin{thm}\label{thm2-01}
Let $p\geq 1$ and $0\leq q<p$, $Q\in\mathcal{S}_{o}^n$,
and $K_i\in\mathcal{K}_{o}^n$ for each $i=0,1,\cdots$.
If $\{\widetilde{C}_{p,q}(K_i,Q,\cdot)\}$ converges to $\widetilde{C}_{p,q}(K_0,Q,\cdot)$ weakly,
then $\{K_i\}$ converges to $K_0$ in the Hausdorff metric.
\end{thm}

\begin{proof}
If the sequence $\{K_i\}$ does not converge to $K_0$,
then there exist $\varepsilon_2 >0$ and a subsequence of $\{K_i\}$, denoted again by $\{K_i\}$, such that
\begin{align*}
||h_{K_i}-h_{K_0}||\geq \varepsilon_2,
\end{align*}
for all $i$.

By Lemma \ref{cdsb-dmvsabb} and the Blaschke's selection theorem,
we have $\{K_{i}\}$ has a convergent subsequence, denoted again by $\{K_i\}$, with $\lim_{i\rightarrow+\infty}K_{i}=K'_0$.
Then, $K'_0\neq K_0$ and $o\in K'_0$ by $K_i\in\mathcal{K}_{o}^n$ for all $i$.
By Lemma \ref{pmixedvolumeb} and the fact that $\{\widetilde{C}_{p,q}(K_i,Q,\cdot)\}$ converges to $\widetilde{C}_{p,q}(K_0,Q,\cdot)$ weakly, we have $\{\widetilde{C}_{p,q}(K_i,Q,S^{n-1})\}$ is bounded from above.
Together with Lemma \ref{Kinteriorp}, we obtain $K'_0\in\mathcal{K}_{o}^n$.

From Lemma \ref{dualcmwc1}, we have
$\{\widetilde{C}_{p,q}(K_i,Q,\cdot)\}$ converges to $\widetilde{C}_{p,q}(K'_0,Q,\cdot)$ weakly.
Since $\{\widetilde{C}_{p,q}(K_i,Q,\cdot)\}$ converges to $\widetilde{C}_{p,q}(K_0,Q,\cdot)$ weakly, then
\begin{align*}
\widetilde{C}_{p,q}(K'_0,Q,\cdot)=\widetilde{C}_{p,q}(K_0,Q,\cdot).
\end{align*}
By Theorem \ref{uniqueness}, we obtain $K'_0=K_0$. This contradicts $K'_0\neq K_0$.

Therefore,
the sequence $\{K_i\}$ converges to $K_0$ in the Hausdorff metric.
\end{proof}

\begin{remark}
The case $q=n$ of Theorem \ref{thm2-01} is due to Zhu \cite{Zhu4}.
Moreover, the case $q=0$ and $Q=B$ of Theorem \ref{thm1-1} gives the continuity of the solution to $L_p$ Aleksandrov problem for $p\geq 1$.
\end{remark}

By the arguments as in the proof of Theorem \ref{thm2-01}, we obtain the following theorems.

\begin{thm}\label{thm2-2}
Let $p\geq 1$ and $0<q_0<p$, $Q\in\mathcal{S}_{o}^n$
and $K_i\in\mathcal{K}_{o}^n$ for each $i=0,1,\cdots$.
If $\widetilde{C}_{p,q_i}(K_i,Q,\cdot)=\widetilde{C}_{p,q_0}(K_0,Q,\cdot)$ with $\lim_{i\rightarrow+\infty}q_i=q_0$,,
then $\{K_i\}$ converges to $K_0$ in the Hausdorff metric.
\end{thm}

\begin{thm}\label{thm2-3}
Let $p_0>1$ and $0<q<p_0$, $Q\in\mathcal{S}_{o}^n$
and $K_i\in\mathcal{K}_{o}^n$ for each $i=0,1,\cdots$.
If $\widetilde{C}_{p_i,q}(K_i,Q,\cdot)=\widetilde{C}_{p_0,q}(K_0,Q,\cdot)$ with $\lim_{i\rightarrow+\infty}p_i=p_0$,
then $\{K_i\}$ converges to $K_0$ in the Hausdorff metric.
\end{thm}

\vskip 0.3 cm

%%%%%%%%%%%%%%%%%%%%%%%%%%%%%%%%%%%%%%%%%%%%%%%%%%%%%%%%%%%%%%%%%%%%%%%%%%%%%%%%%%%%%%%%%%%%%%%%%%%%%%%%%%%%%

\end{document}